\documentclass[twoside,11pt]{amsart}
\usepackage{amsmath}
\usepackage{amsthm}
\usepackage{amssymb}
\usepackage{mathrsfs}
\usepackage{color}
\usepackage{chicagoc}
\usepackage[colorlinks]{hyperref}

\theoremstyle{plain}

\newtheorem*{thma}{Theorem A}
\newtheorem*{thmb}{Theorem B}
\newtheorem*{thmc}{Theorem C}

\newtheorem{thm}{Theorem}[section]

\newtheorem{lem}[thm]{Lemma}

\newtheorem{prop}[thm]{Proposition}
\newtheorem{cor}[thm]{Corollary}

\theoremstyle{definition}

\theoremstyle{remark}
\newtheorem{rem}[thm]{Remark}
\newtheorem*{rem*}{Remark}

\newcommand{\DS}{\displaystyle}
\newcommand{\cas}{\overset{as}{\rightarrow}}

\def\des{\mathop{\hbox{$\longrightarrow$}}\limits}
\def\oL{\overline{L}}
\def\tL{\tilde{L}}
\def\oh{\overline{h}}
\def\s{\sigma}
\def\bl{\begin{lem}}
\def\el{\end{lem}}
\def\bp{\begin{proof}}
\def\ep{\end{proof}}
\def\beq{\begin{eqnarray*}}
\def\eeq{\end{eqnarray*}}
\def\tet{{\theta}}
\def\tx{\text}
\def\V{\hat{V}}
\def\Z{{\mathbb Z}}        
\def\R{{\mathbb R}}        
\def\<{{\langle}}
\def\>{{\rangle}}
\def\P{{\mathbb P}}        
\def\E{{\mathbb E}}        
\def\1{{\mathbf 1}}        
\def\iy{{\infty}}
\def\grad{\mathsf{grad}}
\def\L{{\mathscr L}}
\def\L2{{\mathscr L}^2_{\rho_\X}}
\def\A{\bar{A}}
\def\b{\bar{b}}
\def\w{\bar{w}}
\def\M{M_\rho}

\def\e{{\epsilon}}

\def\H{{\mathscr H}}
\def\W{{\mathscr W}}
\def\P{{\mathscr P}}

\def\X{{\mathscr X}}
\def\Y{{\mathscr Y}}
\def\Z{{\mathscr Z}}
\def\F{{\mathcal F}}

\def\PPi{{\bar{\Pi}}}

\def\Err{{\mathscr E}}
\def\N{{\mathbb N}}
\def\z{{\mathbf z}}
\def\Prob{{\bf Prob}}

\def\grad{{\rm grad}}

\def\supp{{\rm supp}}
\def\span{{\rm span}}

\def\amax{{\overline{\alpha}}}

\def\amin{{\underline{\alpha}} }

\def\t{t_0}
\def\g{\gamma}

\def\al{\alpha}
\def\be{\beta}
\def\la{\lambda}
\def\De{\Delta}
\def\de{\delta}
\def\ka{\kappa}
\def\ga{\gamma}
\def\ze{\zeta}
\def\eps{\epsilon}
\def\ot{\overline{t}}

\newcommand{\Es}{\mathbb{E}}

\begin{document}

\title[Online Learning as Stochastic Approximation of Regularization Paths]{Online Learning as Stochastic Approximation of Regularization Paths: Optimality and Almost-sure Convergence}
\thanks{The research of P.T. was supported in part by a Leverhulme Prize. The research of Y.Y. was supported in part by National Basic Research Program of China (973 Program 2011CB809105, 2012CB825501), NSFC grant 61071157, NSF grant 0325113, and a professorship in the Hundred Talents Program at Peking University. Part of this work was done while both authors visited Toyota Technological Institute at Chicago in the spring of 2005, written in an earlier version in \cite{Yao06,TarYao08}.}

\author{Pierre Tarr\`es}
\address{Pierre Tarr\`es, Mathematical Institute, University of Oxford, 24-29 St Giles', Oxford OX1 3LB,  U.K.}
\email{tarres@maths.ox.ac.uk}

\author{Yuan Yao}
\address{Yuan Yao, School of Mathematical Sciences, Peking University, Beijing, China 100871.}
\email{yuany@math.pku.edu.cn}


\date{January 17, 2013.}


\begin{abstract}
In this paper, an online learning algorithm is proposed as sequential stochastic approximation of a regularization path converging to the regression function in reproducing kernel Hilbert spaces (RKHSs). We show that it is possible to produce the best known  strong (RKHS norm) convergence rate  of batch learning, through a careful choice of the gain or step size sequences, depending on regularity assumptions on the regression function. The corresponding weak (mean square distance) convergence rate  is optimal in the sense that it reaches the minimax and individual lower rates in the literature. In both cases we deduce almost sure convergence, using Bernstein-type inequalities for martingales in Hilbert spaces. 

To achieve this we develop a bias-variance decomposition similar to the batch learning setting; the bias consists in the approximation and drift errors along the regularization path,  which display the same rates of convergence, and the variance arises from the sample error analysed as a reverse martingale difference sequence. 
The rates above are obtained by an optimal trade-off between the bias and the variance.  
\end{abstract}

\maketitle

\bigskip




\section{Introduction}


Consider the following classical problem of learning from examples: given a sequence of i.i.d. random examples
$(z_t=(x_t,y_t))_{t\in \N}$ drawn from a probability measure $\rho$ on $\X\times \Y$, one seeks to approximate
the \emph{regression function}
$$f_\rho(x):=\int_\Y y d \rho_{\Y|x},$$ \emph{i.e.} the conditional expectation of $y$ given $x$. Recall that $f_\rho$ minimizes the following mean square error
\begin{equation}  \label{eq:ls}
\Err(f) = \int_{\X\times \Y} (f(x) - y)^2 d \rho.
\end{equation}

The error of the approximation $f$ of $f_\rho$ is estimated for instance through the norm $\|f -
f_\rho\|_\infty$ or $\|f - f_\rho\|_\rho$, where
$$\|f\|_\rho = \|f\|_{\L2} =\left\{\int_X |f(x)|^2 d
\rho_{\X}\right\}^{1/2},$$ ($\rho_\X$ being the marginal distribution of $\rho$ on
$X$), or through other norms in Hilbert spaces which, as we shall see later, may capture different regularity features of this approximation. 

An {\it online learning algorithm} aims at obtaining 
this approximation of the regression function recursively, using at each time step the new example $z_t=(x_t,y_t)$ to update the current hypothesis $f_{t-1}$ (approximating
$f_\rho$) to $f_t$. In other words, $f_t=T_t(f_{t-1},z_t)$
for some map $T_t:\H\times \X\times\Y\to \H$, where $\H$ is a Hilbert space of functions
from $\X$ to $\Y$, see for example \cite{SmaYao06}. 

On the contrary,  {\it batch learning} algorithms process a sample set  
given once and for all at some fixed time $m$, \emph{i.e.} ${\bf z}=\{(x_i,y_i)\}_{i=1}^m$. 
The classical bias-variance paradigm is that of a trade-off between the requirement to fit the data, \emph{i.e.} to provide a small
empirical error 
$$\hat{\Err}_\z(f):=\frac{1}{m}\sum_{i=1}^m(f(x_i)-y_i)^2,$$
and the size of the space in which $f$ can take place, in order to limit the impact of the noise created by the 
data. For instance, a Tikhonov regularization (or Ridge Regression) procedure as in \cite{EPP00} yields, given
 $\la>0$, 
$$f_{{\bf z},\la}:=\arg\min_{f\in\H}\left\{\frac{1}{m}\sum_{i=1}^m(f(x_i)-y_i)^2
+\la \|f\|_\H^2\right\}.$$
For more background on regularization of inverse problems, see for instance \cite{EngHanNeu96}. In modern statistics, an $L_1$-type regularization called LASSO \cite{Tib96}, is proposed in pursuit of sparsity of $f_\rho$ with respect to certain basis. 

The regularization parameter $\la$ is chosen as a function of the sample size $m$, and of some prior 
knowledge on the regularity of the function $f_\rho$. In this setting, probabilistic upper bounds of $\|f_m-f_\rho\|_\H$ were obtained for instance in 
[\citeNP{CucSma02}, \citeNP{SmaZho-ShannonIII}]. 

In online learning, the sample size $t$ is changing over time, so that the regularization parameter needs to be updated at each time step, and follows the {\it regularization path} defined as follows. 
Let, for all $\la>0$, $f_\la$ be the solution of the regularized least 
square problem
\begin{equation}  \label{eq:reg}
f_\la = \arg \min_{f\in \H} \Err(f) + \la \|f\|_\H^2.
\end{equation}
Depending on assumptions on the Hilbert space $\H$ and on the regularity of $f_\rho$, $f_\la$ converges to $f_\rho$ in $\L2$ or $\H$-norm when $\la \to 0$. The map 
\begin{eqnarray*}
f_.: \R_+  &\longrightarrow &\H\\
\la  &\longmapsto &f_\la
\end{eqnarray*}
is called \emph{regularization path} of $f_\rho$ in $\H$.

Regularization paths gain rising attention from statistics, in particular in the LASSO case \cite{Efron04}, where they are piecewise linear with respect to the parameter, which enables one to track the entire path with nearly the same amount of computational cost as a single fixed regularization. This property generalizes to the case where the loss and the penalty are respectively piecewise quadratic and linear \cite{RosZhu07}; note that this however does not include Tikhonov regularization. 

Our purpose in this paper is to iteratively define an ``online" sequence of functions $(f_t)_{t\in\N}\in \H$, which will provide a stochastic approximation of the Tikhonov regularization path $(f_{\la_t})_{t\in\N}\in \H$. With an adequate choice of the regularization parameters $\la_t\to 0$ based on a bias-variance trade-off, we show such a sequential stochastic approximation  to be \emph{optimal} in the sense that it reaches minimax and individual lower bound rates of convergence.   

Our algorithms can be regarded as stochastic gradient descent algorithms to solve (\ref{eq:reg}) with time varying regularization parameter $\la_t$, an extension from early works \cite{SmaYao06,Yao10} which investigate the convergence  $f_t \to f_\la$ for  fixed $\la_t = \la>0$. In that case a weak probabilistic upper bound for $\|f_t - f_\la\|$ was first proposed in \cite{SmaYao06}, based on Markov's inequality. Improved upper bounds were later obtained in \cite{Yao10}, leading in some cases to the same rate of convergence of $(f_t)_{t\in \N}$ to $f_\la$ as in batch learning given $t$ examples. 

However, as we shall see in this  paper, time-varying $\la_t$  was not addressed so far and leads to a more complicated bias-variance decomposition, whose heuristics is related to the existence of a phase transition in the convergence rate in stochastic approximation. We refer the reader to  [\citeNP{Duflo96}; \citeNP{KusYin03}] for background on stochastic algorithms.

As in previous studies, we choose in this paper the Hilbert space $\H$ to be a reproducing kernel Hilbert space (RKHS) $\H_K$ for some kernel $K$. RKHS enables one to analyze nonparametric regressions in a coordinate-free manner, and the gradient descent method then takes an especially simple form [e.g. \citeNP{KivSmoWil04}].
Moreover, RKHS provides a unified framework in several important settings, e.g.

(i) generalized smooth spline functions in Sobelev spaces \cite{Wahba90}, 

(ii) real analytic functions with bounded bandwidth \cite{Daubechies92}
and their generalizations \cite{SmaZho-ShannonI}, 

(iii) gaussian processes [\citeNP{Loeve48}; \citeNP{Parzen61}]. 

In fact, any Hilbert space of functions on $\X$ with a bounded evaluation functional is a RKHS \cite{Wahba90}. By choosing suitable kernels, $\H_K$ can be used to approximate any function in $\L2$; see for instance \cite{BerTho04} for wider background on RKHS.

Our analysis starts in the setting of a general Hilbert space $\W$  in Section \ref{sec:general}, with the study of  an iteratively defined sequence, which is a stochastic approximation of the solution of some linear equation. This study will be specialized in later sections to the cases $\W=\H_K\tx{ or }\L2$ in order to show the main results of the paper. Two structural decomposition theorems are introduced in that Section \ref{sec:general},  the \emph{reversed martingale
decomposition} and the \emph{martingale decomposition}, and play an important role in the proof of the main results, the former being suitable for strong convergence in $\H_K$ and the latter for weak convergence in $\L2$. 

Both decompositions lead to the breakdown of the total error $f_t - f_\rho$ into four parts: the \emph{initial error} caused by the initial guess $f_0$, the \emph{sample error} as a reverse martingale difference sequence, the \emph{approximation error} $f_{\la_t}-f_\rho$, and the \emph{drift error} along the regularization path $(f_{\la_t})$ caused by time-varying $\la_t$. 

By a suitable choice of step sizes, the initial error won't affect the convergence rates. Now a key observation is that the drift error, which does not appear in previous fixed regularization settings with $\la_t = \la$, has the same order as the approximation error. Bernstein-type inequalities for martingales in Hilbert spaces are then used to bound the sample error. Therefore we have a similar bias-variance decomposition in online learning as in batch learning, with the bias being the approximation and the drift errors, and the variance being the sample error. It is then possible to optimize in order to yield the same optimal rates in online learning as in batch learning.


The main theorems in this paper provide some probabilistic upper bounds for the convergence of $(f_t)_{t\in \N}$ to $f_\rho$, in $\H_K$ or $\L2$,
under the assumption that $f_\rho\in \H_K$ has additional regularity. The convergence rate in $\L2$ is optimal in the sense that it reaches the minimax and individual lower rate. The convergence rate in $\H_K$ meets the same best rates as in batch learning \cite{SmaZho-ShannonIII}. Both upper bounds depend on a logarithmic power $\al>0$ of the confidence threshold $\de$ (i.e. $O(\log^\alpha 1/\de)$). They imply by Borel-Cantelli Lemma the almost sure convergence of $f_t$ to $f_\rho$ in $\H_K$ and $\L2$. Such a theorem improves on our early result in 2006 (see \cite{Yao06}), where in mean square distance the upper bounds depended polynomially on the confidence (i.e. $O(\delta^{-\alpha})$), and whence solves the open problem raised therein.


The  paper is organized as follows. Section \ref{sec:main} collects the main results. Section \ref{sec:general} studies
stochastic approximations of regularization paths for linear operator equations in general Hilbert spaces, where the key martingale and reverse martingale decompositions are presented. Section \ref{sec:drifts} collects some estimates on the drift along the regularization path,
$\|f_{\la} - f_{\mu}\|$ ($\la,\mu>0$), which are needed for the study of the bias, i.e. the approximation and drift errors.
Sections \ref{sec:hk} and \ref{sec:l2} respectively yield upper bounds for convergence in $\H_K$ and $\L2$. Appendix A derives
a probabilistic inequality from the Pinelis-Bernstein inequality for martingales in Hilbert spaces, which is used to derive the probabilistic upper bounds in this paper. Appendix B collects some preliminary upper bounds used in the paper. Appendix C gives proofs of some results in Section \ref{sec:thmconv}.

\section{Main Results} \label{sec:main}

\subsection{Notations and Assumptions}
Let $\X\subseteq \R^n$ be closed, $\Y=\R$ and $\Z=\X\times \Y$.
Let $\rho$ be a probability measure on $\Z$, $\rho_{\X}$ be the induced marginal probability measure on $\X$,
and let $\rho_{\Y|x}$ be the conditional probability measure on $\Y$ with respect to $x\in
X$. Define $f_\rho:\X\to \Y$ by $f_\rho (x) = \int_\Y y d \rho_{\Y|x}$,
the \emph{regression function of $\rho$}. In the sequel, we let $\E[\cdot]$ be the expectation with respect to $\rho$.

Let $\L2$ be the Hilbert space of
square integrable functions with respect to $\rho_\X$. In the sequel $\|\ \|_\rho$
denotes the norm in $\L2$.

Let $K:\X\times \X\to \R$ be a \emph{Mercer kernel}, i.e. a
continuous\footnote{In computer science literature, one often bears in mind some implicit feature map $\Phi:\X\to\H$ which takes an input vector $x$ to a high (or infinite) dimensional feature vector, say an element of a Hilbert space $\H$, and then one considers explicitly the inner product $K(x,x')=\<\Phi(x),\Phi(x')\>$ as the kernel. In this construction, the continuity of $K$ is equivalent to continuity of the feature map $\Phi$.} symmetric real function which is \emph{positive
semi-definite} in the sense that $\sum_{i,j=1}^m c_i c_j K(x_i,
x_j)\geq 0$ for any $m\in \N$ and any choice of $x_i\in X$ and
$c_i \in \R$ ($i=1,\ldots,m$). A Mercer kernel $K$ induces
a function $K_x : \X\to \R$ ($x\in \X$) defined by $ K_x(x^\prime) = K(x,x^\prime)$.
Let $\H_K$ be the \emph{reproducing kernel Hilbert space} (RKHS)
associated with a Mercer kernel $K$, i.e. the completion of the $\span\{K_x: x\in \X\}$
with respect to the inner product, defined as the linear extension of the
bilinear form $\<K_x, K_{x^\prime}\>_K = K(x,x^\prime)$ ($x,x^\prime \in \X$). The norm
of $\H_K$ is denoted by $\|\ \|_K$. The most important property of RKHS is the \emph{reproducing property}:
for all $f\in \H_K$ and $x\in X$, $f(x)=\<f,K_x\>_K$.

Throughout this paper, assume that

\noindent {\bf Finiteness Condition.}
\noindent (A) There exists a constant $\ka\geq 0$ such that
\begin{equation*}
\kappa:=\sup_{x\in \X} \sqrt{K(x,x)} < \infty.
\end{equation*}
\noindent (B) There exists a constant $\M \geq 0$ such that
\[ \supp(\rho)\subseteq \X\times [-\M, \M]. \]

Define the linear map
$$L_K:\L2\to \H_K$$ 
by the following integral transform
$$L_K(f)(x):= \int_X K(x,t) f(t) d\rho_X(t).$$ 
It is well-known that $L_K$ is well-defined, and that composition with the inclusion $\H_K \hookrightarrow\L2$ yields a compact positive self-adjoint operator on $\L2$ [e.g. \citeNP{HalSun78}, \citeNP{CucSma02}]. The restriction 
$L_K|_{\H_K}:\H_K \to \H_K$ is the \emph{covariance operator} of $\rho_\X$ in $\H_K$: by the reproducing property, 
$$L_K|_{\H_K} = \E [\<\cdot,K_x\>K_x].$$ Abusing notation, we will denote the three operators by $L_K$ in the sequel.

Note that, by Cauchy-Schwarz inequality, $\|L_Kf\|_{\iy}\le\ka^2\|f\|_{\L2},$ so that 
\begin{equation}
\label{ub-lk}
\|L_K\|_{\L2\to\L2}\le\ka^2.
\end{equation}

The compactness of $L_K:\L2\to \L2$ implies the existence of an orthonormal eigensystem $(\mu_\al,\phi_\al)_{\al\in \N}$ in $\L2$. 
Recall that (see [\citeNP{CucSma02}] for instance)
$$\sum_{\al\in\N}\mu_\al=\int_\X K(x,x)d\rho_\X(x)\le\ka^2.$$

We assume in this paper that all eigenvalues $\mu_\al$ are positive. We can define, for all $r>0$, 
\begin{equation} \label{eq:LKr}
\begin{array}{rcl}
L_K^r: & \L2 & \to \L2 \\
& \DS \sum_{\al\in\N} a_\al \phi_\al & \DS \mapsto \sum_{\al\in\N} a_\al \mu_\al^r \phi_\al;
\end{array}
\end{equation}
$L_K^{r}$ can be regarded as a low-pass filter, and $\|L_K^r\|=\max_{\al\in\N} \mu_\al^r=\|L_K\|^r$. Note that $L_K^{1/2}:\L2\to \H_K$ is an isometrical isomorphism of Hilbert spaces.  Hence the eigenfunctions $(\phi_\al)_{\al\in\N}$ are orthogonal both in $\L2$ and $\H_K$. 

For all $f\in\L2$ and $r>0$, we write $L_K^{-r} f \in \L2$ when $f$ lies in the image of the mapping $L_K^r:\L2\to \L2$. Note that, if $r\geq 1/2$, then this implies $f\in \H_K$ because of the isometry $L_K^{1/2}$ between $\L2$ and $\H_K$.

For all $\la>0$ and $r\in\R\setminus\{0\}$, we can similarly define $(L_K+\la I)^r:\L2\to\L2$, which is a bijection, since 
 $\sum_{\al\in\N} a_\al^2<\iy$ is equivalent to $\sum_{\al\in\N} a_\al^2(\la+\mu_\al)^r<\iy$, using $\mu_\al\to_{\al\to\iy}0$.

It can be shown [e.g. \citeNP{CucSma02}] that for any $\la\in \R_+$, the solution of (\ref{eq:reg}) is
\begin{equation}
\label{expfla}
f_{\la} = (L_K + \la I)^{-1} L_K f_\rho \in \H_K.
\end{equation}

In this paper, by $B_1,C_1,D_1, B_2,C_2,D_2,\ldots$, we denote various constants, which are defined ``locally'' in the sense 
that the same notations appeared in different sections has different meanings. 

\subsection{Stochastic Gradient Algorithms}
Let $\F=(\F_t)_{t\in\N_0}\in \X\times \Y$ be the filtration
$\F_t=\sigma\{(x_i,y_i)~:~1\le i\le t\}$. In the sequel denote by $\E_t=\E[\cdot|\F_t]$, the conditional expectation w.r.t. $\F_t$. Consider the following $\F_t$-adapted process $(f_t)_{t\in \N}$ taking values in $\H_K$,
\begin{equation}\label{eq:ft}
f_t = f_{t-1} - \gamma_t [(f_{t-1}(x_t)-y_t)K_{x_t} +\lambda_t f_{t-1}], \ \ \ \ \ \mbox{for some fixed $f_0\in \H_K$, e.g. $f_0:=0$}
\end{equation}
where \\
(I) for each $t$, $(x_t,y_t)$ is independent and identically distributed (i.i.d.) according to $\rho$; \\
(II) the {\it gain (step size) sequence}  $(\gamma_t)_{t\in\N}$ and {\it regularization sequence} $(\lambda_t)_{t\in\N}$ are taking values in $\R_+:=(0,\infty)$, and converging to $0$ as $t$ goes to infinity. 

\begin{rem}
\label{rem:ft}
The computational cost of this algorithm typically is $O(t^2)$. As each step $t$, the main computational cost is due to the evaluation $f_{t-1}(x_t)$ which
needs to access all $K_{x_i}$ ($1\leq i\leq t$) in $O(t)$ steps. Thus the total cost is of $O(t^2)$ at time $t$. In the cases that one can store and access
the values $f_t(x)$ for all $x$, e.g. on a grid of $\X$, the computational cost is merely linear $O(t)$ at the requirement of large memory and fast memory
access.
\end{rem}

By reproducing property, we can see that the gradient map of
$$V_z(f)= \frac{1}{2}[(f(x)-y)^2 + \la \|f\|_K^2], \ \ \ \ z=(x,y)\in \Z$$
is given by $\grad V_z(f)=(f(x)-y)K_{x} +\lambda f$ [e.g. \citeNP{SmaYao06}], as a random variable depending on $z$.
Since the expectation $\E[V_z(f)]=2(\Err(f)+\la\|f\|_K^2)$, algorithm (\ref{eq:ft}) can thus be regarded as
stochastic approximations of gradient descent method to solve (\ref{eq:reg}), for each $\la=\la_t$.

\subsection{Main Theorems}
Theorem A provides sufficient conditions for the convergence  of the {\it online learning} sequence $(f_t)_{t\in\N_0}$ in (\ref{eq:ft}) to the regression function $f_\rho$. Theorem B and C explicit the corresponding convergence rates, respectively in $\H_K$ and $\L2$.

\medskip

\begin{thma} [Sufficient conditions for convergence]
Assume $f_\rho\in\H_K$, and let $(f_t)$ be defined by equation (\ref{eq:ft}), with assumptions (I)-(II).   Then
$$\limsup_{t\to\iy}\Es[\|f_t-f_\rho\|^2_K]=0,$$
if the following conditions are satisfied: 

\noindent (A) $\DS \sum_{t\to \infty} \gamma_t \la_t = \infty$.

\noindent (B) $\DS \limsup_{t\to \infty} \sum_{k=1}^n\g_k^2\prod_{i=k+1}^n(1-\g_i\la_i)^2 =0$,

\noindent (C)  $\DS \limsup_{t\to \infty} \sum_{k=1}^n \|f_{\la_k}-f_{\la_{k-1}}\|_K\prod_{i=k+1}^n(1-\g_i\la_i)=0$.
\end{thma}

This theorem will be proved in Section \ref{sec:general}, as a consequence of Theorem \ref{thm:convergence} in the setting of  Hilbert spaces. Assumptions $(B)$ and $(C)$ can be replaced by the stronger (but less technical) assumptions $(B')$ and $(C')$ in Corollary \ref{cor:convergence} that $\g_t/\la_t\to0$ and $\|f_{\la_t} - f_{\la_{t-1}}\|_K/(\la_t \gamma_t)\to0$.

\begin{rem}
Although $\la_t \to 0$, condition (A) puts a restriction that $\ga_t\la_t$ can not drop too fast, in fact this is necessary to ``forget'' the error
caused by the initial guess $f_0$. Condition (B) says that the step size $\ga_t \to 0$, and it has to drop faster than the regularization parameter $\la_t$.
Such a condition is to attenuate the random fluctuation caused by sampling.
Condition (C) implies that the drifts of the regularization path $(f_{\la_t})$ converges to zero, at a speed faster than $\ga_t \la_t$. This condition
says that in the long run, the drifts along the regularization path should decrease fast enough for the algorithm to follow the path. The drifts depend on regularity of $f_\rho$, that the smoother $f_\rho$ is, the faster drifts go down. 
\end{rem}

In the next two theorems (B) and (C) we choose the sequences $(\g_t)_{t\in\N}$ and $(\la_t)_{t\in\N}$ in order to optimize the rates of convergence  in $\H_K$ and $\L2$. This optimization is twofold. 

First, the study of convergence of approximations of ordinary differential equations generically yields a phase transition between a slower rate with ``shadowing'' of mean-field trajectories, and a faster one, normally distributed after renormalization. Even though the picture is more complicated in our case, in particular because the vector $f_t$ is infinite-dimensional, this  justifies here that we choose $\g_t\la_t$ reciprocally linear in $t$. 

Second, optimization over $(\g_t)$ at fixed  $(\g_t\la_t)$ yields a bias-variance trade-off similar to the one observed in statistical ``batch'' learning, which relies on the regularity assumption on the regression function $f_\rho$.

More precisely, let us first recall the phase transition in classical finite-dimensional stochastic approximation, in the rate of convergence towards a stable equilibrium. Naturally, we study the projections of the algorithm on the base of eigenvectors of the linearization of the ordinary differential equation at the equilibrium. Let $(\eta_t)_{t\in\N}$ be one of these projections, and assume for instance that the corresponding eigenvalue is $-1$, so that the stochastic recursion is of the form
$$\eta_{t+1}=\eta_t+\g_t(-\eta_t+\e_{t+1}+r_{t+1}),$$
where $\Es_{t-1}[\e_t]=0$, $(\e_t)$ is bounded, and $(r_t)$ is small. For simplicity we will assume that $r_t=0$ (which corresponds to the special case $\la_t=\la$ is a constant), but the heuristics holds on to the general case where $r_t$ is less than quadratic in all coordinates. Let, for all $t\in\N$, $\be_t:=\prod_{k=1}^t(1-\g_k)$. Then it is easy to show by induction that 
$$\eta_t=\be_t\left[\eta_0+\sum_{j=1}^t\frac{\g_j}{\be_j}\e_j\right].$$

Now suppose for instance that $\g_t\sim c/t$ ($c>0$); then $\be_nn^{c}\des_{n\to\iy}C>0$. Depending on the choice of $c$, $\eta_t$ exhibits the following phase transition at $c=1/2$ in its asymptotic dynamics.  

\begin{itemize}
\item If $c<1/2$ then $\sum(\g_j/\be_j)^2<\iy$, therefore  $\sum_{j=1}^t\g_j\e_j/\be_j$ converges a.s. by Doob's convergence theorem, which implies that $\eta_tt^{c}\des_{t\to\iy}C'$ (where $C'$ is a positive random variable), in other words that $(\eta_t)$ asymptotically ``shadows'' one solution of the ODE $$\frac{dx}{dt}=-cx.$$ 
\item On the contrary, if $c>1/2$, then  $\sum(\g_j/\be_j)^2=\iy$, and by the martingale convergence theorem (see for instance \cite{Williams91}), assuming for instance $\Es_{t-1}[\e_t^2]=D^2>0$ constant, and $\eta_t\sqrt{t}$ converges towards a centered normally distributed random variable with variance $a^2D^2/(2a-1)$, and follows an associated Ornstein-Uhlenbeck process, see \cite{Duflo96} for instance.
\end{itemize} 


Therefore it suffices to choose $c>1/2$ to achieve fast convergence rates. In this paper we will set $c=1$ and choose $\ga_t \la_t \sim1/ t$ to meet the heuristics above. 

The next two theorems present some probabilistic upper bounds which characterize the convergence rates in $\H_K$ and $\L2$, under certain regularity assumptions
on the regression function $f_\rho$. 

Let $\t>0$ and, for all $t\in\N$, 
$$\ot:=t+\t,$$
where $\t$ is large enough which won't affect the speed of convergence.
We assume, in the statement of Theorems B and C, that, for all $t\in\N$, 
$$\ga_t = a\left(\frac{1}{\ot}\right)^{\frac{2r}{2r+1}},\,\,\la_t = \frac{1}{a}\left(\frac{1}{\ot}\right)^{\frac{1}{2r+1)}}.$$

\begin{thmb} [Upper Bounds for $\H_K$-convergence]
Assume $L_K^{-r} f_\rho\in \L2$ for some $r\in (1/2,3/2]$,  $a\ge1$,  and $\t^\tet\geq a\ka^2+1$.
Then, for all $t\in \N$, with probability at least $1-\delta$,
\[ \|f_t - f_\rho \|_K \leq  \frac{C_0}{\ot} + \left( C_1a^{1/2-r}\log\frac{2}{\delta} + C_2a \right)
\left(\frac{1}{\ot}\right)^{\frac{2r-1}{4r+2}},\]
where
\[ C_0:=2 \t^{\frac{4r+3}{4r+2}}  M_\rho,\,\,\, C_1:=\frac{20r-2}{(2r-1)(2r+3)}\| L^{-r}_K f_\rho \|_\rho,\,\,\,
C_2:=\frac{20 (\ka+1)^2M_\rho}{\ka}.\]
\end{thmb}

Its proof is given in Section \ref{sec:hk}.
\begin{rem}
Given $\de>0$, $M_\rho$ and $\|L_K^{-r}f_\rho\|_\rho$, one can optimize $a$ in order to minimize
$$h(a):=C_1a^{1/2-r}\log\frac{2}{\delta} + C_2a.$$
This yields the choice $a^*:=[C_1/((r-1/2)C_2)]^{(r+1/2)^{-1}}$, with
$$h(a^*)=\left( r+ 1/2\right) \left[\frac{C_1C_2^{r-1/2} }{(r-1/2)} \right]^{(r+1/2)^{-1}}  .$$
This asymptotic rate in $M_\rho^{(2r-1)/(2r+1)}\|L_K^{-r}f_\rho\|_\rho^{2/(2r+1)} t^{-(r-1/2)/(4r+2)}$ is the same as the best known rates in batch learning algorithms; see [Theorem 2, \citeNP{SmaZho-ShannonIII}].
\end{rem}
\begin{rem}
Note that the upper bound consists of three parts.
The first term at a rate $O(t^{-1})$, captures the influence of the initial choice $f_0=0$, which does not depend on $r$ and is faster than the remaining terms.
The second term at a rate $ O(\| L^{-r}_K f_\rho \|_\rho t^{-(2r-1)/(4r+2)})$, collects contributions from both drifts along the
regularization path $f_{\la_t}-f_{\la_{t-1}}$ and the approximation error $f_{\la_t}-f_\rho$, since they share the same rates up to different constants.
The third term at a rate $O(t^{-(2r-1)/(4r+2)})$, reflects the error caused by random fluctuations by the i.i.d. sampling. Later as we will see, the second term is a bound on the bias and the third term is a bound on the variance.
\end{rem}

\begin{thmc}[Upper Bounds for $\L2$-convergence]
Assume that $L_K^{-r} f_\rho\in \L2$ for some $r\in [1/2,1]$. Assume $a\ge4$, and  $\t^\tet\geq 2+8\ka^2 a$.

Then, for all $t\in \N$, with probability at least $1-\delta$ ($\delta\in (0,1)$),
$$\|f_t - f_\rho \|_\rho \leq
\frac{D_0}{\ot}+\left(D_1a^{-r}+\sqrt{a}D_2\log \frac{2}{\de}\right)\left(\frac{1}{\ot}\right)^{\frac{r}{2r+1}}+\left(a^{3/2}D_3\sqrt{\log \ot}+a^{5/2}D_4\right)\left(\log (2/\de)\right)^2\left(\frac{1}{\ot}\right)^{\frac{4r-1}{4r+2}}.$$
$$D_1:=\frac{5r+1}{r(1+r)}\|L_K^{-r}f_\rho\|_\rho,\,\,\,D_2:=10\ka\M,\,\,\,\,D_3=63\ka^2\M,\,\,\,\,
D_4:=50\ka^2\M\t^{1/2-\tet}.$$
\end{thmc}

Its proof will be given in Section \ref{sec:l2}.

\begin{rem}
When $r\in (1/2,1]$, the first term of $O(1/t)$ and the third term of $O(t^{-\frac{2r-1/2}{2r+1}}\log^{1/2} t)$ both drop faster than the second term of $O(t^{-\frac{r}{2r+1}} )$, whence they can be ignored asymptotically. The second term as the dominant one, roughly speaking has contributions from two parts: the one with constant $D_1$ comes from the bias, \emph{i.e.} the approximation and the drift errors, while the other with constant $D_2$ comes from the variance, \emph{i.e.} the sample error. 
\end{rem}

\begin{rem}
A special case is $r=1/2$, which is equivalent to say $f_\rho \in \H_K$. In this case $\ga_t = \la_t = \ot^{-1/2}$, whence it
does not satisfy the Path Following Condition (B) in Theorem A. But Theorem C suggests a weaker notion that $f_t$ follows the regularization path,
\emph{i.e.} $f_t \to f_{\rho}$ in $\L2$ rather than $\H_K$, which in fact converges at a rate of $O(t^{-1/4} \log^{1/2} t)$ uniformly for all $f_\rho \in \H_K$.
\end{rem}

\begin{rem} 
In all, the convergence in $\L2$ has rates $O(t^{-r/(2r+1)}\log^{1/2} t \cdot \log^2 1/\delta)$, a logarithmic polynomial on $\delta$, whence the Borel-Cantelli Lemma implies almost sure convergence $\|f_t - f_\rho\|_{\L2} \cas 0$. 
This result solves an open problem raised early in \cite{Yao06}.   
\end{rem}

\begin{rem}
To see the asymptotic optimality, consider the generalization error $\Err(f) - \Err(f_\rho)=\|f - f_\rho \|_\rho^2$ [e.g. see \citeNP{CucSma02}]. Since the rate $O(t^{-r/(2r+1)})$ dominates when $r>1/2$, then 
under the same condition of Theorem C, there holds with probability at least $1-\delta$ ($\delta\in (0,1)$), for all $t\in \N$,
\[ \Err(f_t) - \Err(f_\rho) \leq O(t^{-2r/(2r+1)}). \]
For $r\in (1/2,1]$, the asymptotic rate $O(t^{-2r/(2r+1)})$ has been shown to be optimal in the sense that it reaches the minimax and individual lower rate [\citeNP{CapDeV05}]. To be precise, let $\P(b,r)$ ($b>1$ and $r\in (1/2,1]$)
be the set of probability measure $\rho$ on $\X\times \Y$, such that: (A) almost surely $|y|\leq \M$;
(B) $L_K^{-r} f_\rho \in \L2$; (C) the eigenvalues $(\mu_n)_{n\in\N}$ of $L_K:\L2\to\L2$, arranged in a nonincreasing order, are subject to the decay $\mu_n= O(n^{-b})$.
Then the following minimax lower rate was given as Theorem 2 in \cite{CapDeV05},
\[ \liminf_{t\to \infty} \inf_{(z_i)_1^t \mapsto f_t } \sup_{\rho\in \P(b,r)} \Prob \left\{(z_i)_1^t \in \Z^t: \Err(f_t) - \Err(f_\rho) >
C t^{-\frac{2rb}{2rb+1}} \right\} = 1 \]
for some constant $C>0$ independent on $t$, where the infimum in the middle is taken over all algorithms as a map
$\Z^t \ni(z_i)_1^t \mapsto f_t\in \H_K$. 

Note that in the minimax lower rate, the probability measure may change for different data size $t$, which violates the fundamental identical distribution assumption in learning. Therefore \cite{GKKW02} suggests a kind of individual lower rates for learning problems. The following individual lower rate was obtained as Theorem 3 in \cite{CapDeV05}: for every $B>b$,
\[\inf_{((z_i)_1^t\mapsto f_t)_{t\in \N}} \sup_{\rho \in \P(b,r)} \limsup_{t\to \infty} \frac{\E [\Err(f_t)] - \Err(f_\rho)}{t^{-\frac{2rB}{2rB+1}}} > 0, \]
where the infimum is taken over arbitrary sequences of functions $f_t:\Z^t\to \H_K$. It can be
seen that the key difference in the individual lower rate, lies in that by putting $\limsup_{t\to \infty}$ before $\sup_{\rho \in \P(b,r)}$, the probability measure $\rho$ is applied to all sufficiently large $t$. 

Now we compare these lower rates to our upper bound.
Since $L_K:\L2\to\L2$ is a trace-class operator, its eigenvalues are summable. Therefore by taking $b=B=1$, one may obtain an
eigenvalue-independent lower rate $O(t^{-2r/(2r+1)})$ for all possible $L_K$. Therefore, the upper bound by Theorem C reaches
both the minimax and the individual lower rates.
\end{rem}

\section{Sequential Stochastic Approximations of Regularization Paths in Hilbert Spaces} \label{sec:general}
In this section, we study some stochastic approximation sequences in the more general setting of general Hilbert spaces. 

Let $\W$ be a Hilbert space with inner product $\langle~,~\rangle$ and associated norm $\|u\|:=\sqrt{\langle u,u\rangle}$, and let  $SL(\W)$  be the vector space of self-adjoint bounded linear operators on $\W$, endowed with the canonical norm 
$$\|A\|:=\sup_{\|x\|\le1}\|Ax\|.$$

Let $\X$ and $\Y$ be two topological spaces (on which we make no other assumption), let $\Z:=\X\times\Y$ and let $\rho$ be a probability measure on the Borel $\sigma$-algebra of $\Z$.  Let $A: \Z\to SL(\W )$ and $b:\Z\to\W$ be random variables on the sample space $\Z$ taking values 
respectively in $SL(\W )$ and $\W$, and let
$$\A:=\E[A],~\b:=\E[b]$$
be their expectations on $(\Z,\rho)$.

Now assume that $\A$ is a positive operator, hence invertible, but that it has an \emph{unbounded} inverse. Knowing $A$ and $b$, but not $\rho$ (and subsequently not $\A$ and $\b$), and assuming $\b\in \A(\W)$, the aim is to devise a stochastic  algorithm  approximating the solution $\w$ of the following linear equation
\begin{equation} \label{eq:hatlin}
\A w = \b,
\end{equation}
using as data an i.i.d sequence $(z_t)_{t\in\N}$ in $\Z$ with probability law $\rho$. As in the standard setting of Robbins-Monro (see \cite{RobMon51}, \cite{KiefWolf52}), it is natural to consider a stochastic gradient descent algorithm. 

More precisely, the search for the solution $\w$ of (\ref{eq:hatlin}) is equivalent to the minimization of the quadratic potential map $\V:\W\to\R$ 
$$\V(w):=\frac{1}{2}\langle\A(w-\w),w-\w\rangle,$$
whose gradient $\grad ~\V: \W\to\W$ is given by 
$$\grad ~\V(w)=\A w-\b=\Es[Aw-b].$$

In the context of online learning presented in the first two sections, $\W:=\H_K$, 
$A((x,y))(f):=f(x)K_x$, $b((x,y)):=yK_x$ (see Section \ref{sec:online}), so that $\A=L_K$, $\b=L_Kf_\rho$ and $\w=f_\rho$, and $\V(w)=\|f-f_\rho\|_{\L2}=\Err(f) - \Err(f_\rho)$ is the generalization error. 

A natural Robbins-Monro gradient descent algorithm would be 
\begin{equation}
\label{stoch-hilb}
w_{t}= w_{t-1} - \gamma_{t} ( A(z_t)w_{t-1} - b(z_t) ), 
\end{equation}
since $\E_{z_t\sim\rho}[A(z_t)w_{t-1} - b(z_t)]=\A w_{t-1}-\b$. 

However, the sample complexity analysis on Hilbert spaces, in order to estimate the sample size sufficient to approximate the minimizer with high probability, requires boundedness of $\A^{-1}$ (see for instance \cite{SmaYao06}). 

To solve this ill-posed problem with unbounded $\A^{-1}$, one may construct sequences of random variables $(A_t)_{t\in\N}$ and $(b_t)_{t\in\N}$ on the sample space $\Z$ taking values respectively in $SL(\W)$ and $\W$, with the assumption that, if
$$\A_t:=\E[A_t],~~\b_t:=\E[b_t]$$ are their expectations on $(\Z,\rho)$,  
$\A_t$ has bounded inverse and  $\A_t\to\A$, $\b_t\to\b$. Then the aim is to find assumptions ensuring that the stochastic approximation sequence $(w_t)_{t\in\N}$ iteratively defined by $w_0:=W_0$ deterministic, and
\begin{equation}
\label{defwt}
w_{t}= w_{t-1} - \gamma_{t} ( A_t(z_t) w_{t-1} - b_t(z_t)),
\end{equation}
where $(\g_t)_{t\in\N}$ is a real positive sequence, converges to the solution $\w$ of (\ref{eq:hatlin}) as $t$ goes to infinity.

This question can be divided into two subquestions: first the {\it deterministic } convergence of 
\begin{equation}
\label{defwwt}
\w_t:=\A_t^{-1} \b_t.
\end{equation}
 to $\w$, the path $t\mapsto\w_t$ being then called a \emph{regularization path} of the solution of equation (\ref{eq:hatlin}), and second the {\it probabilistic }convergence of the quantity
\begin{equation}
\label{defrt}
r_t:=w_t - \w_t,
\end{equation}
which we call the \emph{remainder} (note that $\w_t=\A_t^{-1} \b_t\not=\E[w_t]$ in general). 
In the online learning case (see Section \ref{sec:online}), we choose $A_t:=A+\la_t I$, $(\la_t)_{t\in\N}$ positive sequence, $b_t:=b$, so that $\w_t=f_{\la_t}\to f_\rho$ in $\H_K$.

We provide in Section \ref{sec:decomp} two structural decompositions of $r_t$, respectively a reversed martingale  and a martingale one. Both expand $r_t$ into three parts: one depending on the initial value of $r_.$  called the \emph{initial error}, one depending on the \emph{drift} 
\begin{equation}
\Delta_j:=\w_j - \w_{j-1}
\end{equation}
along the regularization path $(\w_t)$ called the \emph{drift error}, 
and finally one random variable of zero mean called the \emph{sample error}, respectively written as a reversed martingale and 
as a martingale at time $t$. 

The reversed martingale decomposition will, on one hand, enable us to prove Theorem \ref{thm:convergence} below,  whose corollary is Theorem A in the context of online learning, and which provides sufficient assumptions on the asymptotic behaviour of the norms of $A_t$, $A_t^{-1}$, $\A^{-1}_t$ and $b_t$ for the convergence of the variance of the remainder $r_t$. On the other hand, this reversed decomposition will yield Theorem B giving upper bounds on $f_t-f_\rho$ in $\H_K$ with high probability, proved in Section \ref{sec:hk}. 

The martingale decomposition will imply Theorem C giving upper bounds of $f_t-f_\rho$ in $\L2$ with high probability, proved in Section \ref{sec:l2}.

\subsection{Two Structural Decomposition Theorems}
\label{sec:decomp}
For all $j$, $t$ $\in\N$, let $\Pi_j^t$ be the {\it random }operator on $\W$, on the sample space $\Z^\N$, defined by
\begin{equation*} \Pi_j^t((z_i)_{i\in\N})=
\left\{
\begin{array}{lr}
\DS \prod_{i=j}^t \left( I - \gamma_i A_{i}(z_i) \right) &\tx{ if }j\leq t; \\
I &\tx{ otherwise.}
\end{array}
\right.
\end{equation*}

By a slight abuse of notation, we let $A_t:=A_t(z_t)$ and $b_t:=b_t(z_t)$ in the sequel, when there is no ambiguity. 
\begin{thm}[Reversed Martingale Decomposition] \label{thm:rmart} For all $s$, $t$ $\in \N$, $t\ge s$,
\begin{equation}  \label{eq:rmart}
r_{t} = \Pi_{s+1}^t r_{s} - \sum_{j=s+1}^t \gamma_j \Pi_{j+1}^t (A_j \w_j - b_j) - \sum_{j=s+1}^t \Pi_{j}^t \Delta_j
\end{equation}
\end{thm}

\begin{rem}
Note that $\Pi_{j+1}^t$ is an operator whose randomness only depends on $z_{j+1},\ldots z_t$, whereas the randomness in $A_j \w_j - b_j$, with zero mean,  only depends on $z_j$.
By independence of $z_t$, $t\in \N$, the conditional expectation $\E[\gamma_j \Pi_{j+1}^t (A_j \w_j - b_j)|z_{j+1},\ldots,z_{t}]$ is $0$, whence
for each $t$, $\ga_j \Pi_{j+1}^t(A_j\w_j - b_j)$ is a \emph{reversed martingale difference sequence} whose sum is a \emph{reversed martingale sequence} with zero mean.
For more background on reversed martingales, see for example \cite{Neveu75}. 
\end{rem}

\begin{proof}[Proof of Theorem \ref{thm:rmart}] By definition,
\begin{eqnarray*}
r_{t} & = & w_{t}-\w_{t} \\
& = & w_{t-1}- \w_{t} - \gamma_t (A_t w_{t-1} - b_t) \\
& = & (I - \gamma_t A_t)(w_{t-1} - \w_{t-1})  - (I-\gamma_t A_t) (\w_t-\w_{t-1})+\gamma_t A_t(w_{t-1}-\w_t)- \gamma_t (A_t w_{t-1} - b_t)\\
& = & (I - \gamma_t A_t)(w_{t-1} - \w_{t-1})  - (I-\gamma_t A_t) (\w_t-\w_{t-1})- \gamma_t (A_t \w_t -b_t) 
\end{eqnarray*}
which implies
\begin{equation}
 r_{t} = (I-\gamma_t A_t) r_{t-1} - \gamma_t (A_t \w_t - b_t) - (I-\gamma_t A_t)\Delta_t .
\end{equation}
The result follows by induction on $t\in \N$, $t\ge s$.
\end{proof}
For all $j$, $t$ $\in\N$,  let 
$$\chi_t=(\A_t - A_t) w_{t-1} +(b_t - \b_t),$$ 
and  let $\PPi_j^t$ be the {\it deterministic} operator on $\W$ defined by
\begin{equation*}
\PPi_j^t=
\left\{
\begin{array}{lr}
\DS \prod_{i=j}^t \left( I - \gamma_i  \A_i \right) & \tx{ if }j\leq t; \\
I, & \tx{ otherwise.}
\end{array}
\right.
\end{equation*}
\begin{thm}[Martingale Decomposition] \label{thm:mart} 
For all $s$, $t$ $\in \N$, $t\ge s$,
\begin{equation} \label{eq:mart}
r_{t} = \PPi_{s+1}^t r_{s} + \sum_{j=s+1}^t \gamma_j
\PPi_{j+1}^t \chi_j - \sum_{j=s+1}^t \PPi_{j}^t \Delta_j
\end{equation}
\end{thm}

\begin{rem}
The martingale decomposition was proposed in \cite{Yao10}.
Contrary to the reversed martingale decomposition, only the \emph{sample error} is random here, the operator $\PPi_{j+1}^t$ being deterministic. The process $(\gamma_j \PPi_{j+1}^t \chi_j)_{j\in\N}$ is a
\emph{martingale difference sequence} since, for all $j\in\N$ and $t\ge j$, $\E_{j-1}[\gamma_j \PPi_{j+1}^t \chi_j]=0$. Note that the martingale property continues to hold for dependent sampling $z_t(z_1,\ldots,z_{t-1})$, as long as $\Es_{t-1}[A_t(z_t)]=\A_t$ and $\Es_{t-1}[b_t(z_t)]=\b_t$.

The non-randomness of the operator $\PPi_j^t$ will play a key role in the proof of Theorem C in the online learning context, since it will enable us to make explicit calculations involving the spectral decomposition of  $L_K:\L2\to \L2$ (recall that $\A_i = L_K+\la_i$ then).
However, the fact that $\chi_t$, contrary to $A_t \w_t -b_t$ in the reversed expansion, does not depend only on $z_t$ but rather on the whole past $(z_i)_{0\le i \le t}$, makes it necessary to obtain a preliminary upper bound of $\chi_t$ in Appendix C,  which explains the factor $(\log2/\delta)^2$ in Theorem C, rather than $\log2/\delta$ in Theorem B.
\end{rem}

\begin{proof}[Proof of Theorem \ref{thm:mart}] By definition,
\begin{eqnarray*}
r_{t} & = & w_{t}-\w_{t} \\
& = & w_{t-1} - \w_{t}- \gamma_t (A_t w_{t-1} - b_t)  \\
& = &  (I - \gamma_t \A_t)(w_{t-1} - \w_{t})+\g_t\chi_t,\tx{ using }\b_t=\A_t\w_t\\
& = & (I - \gamma_t \A_t)r_{t-1}+ \g_t\chi_t-(I-\gamma_t \A_t) [\w_t - \w_{t-1}].
\end{eqnarray*}
The result follows by  induction on $t\in \N$, $t\ge s$.
\end{proof}

\subsection{Sufficient Conditions for the Convergence of the Remainder}
\label{sec:thmconv}
The following Theorem \ref{thm:convergence}, which implies Theorem A in the context of online learning (see Section \ref{sec:online}), states the convergence of $\|r_t\|^2=\|w_t-\w_t\|^2$ to zero in expectation, under some assumptions on the asymptotic behaviour of the gain sequence $\g_t$ and of the norms of $b_t$ and operators $A_t$, $A_t^{-1}$ and $\A_t^{-1}$. 

The corresponding  {\it Generalized Finiteness Condition} on the asymptotic behaviour of $A_t$ and $b_t$ is a generalization of the {\it Finiteness Condition} in \cite{SmaYao06}. 

\noindent {\bf Generalized Finiteness Condition.} Let  $(\amin_t)_{t\in\N}$  and  $(\amax_t)_{t\in\N}$ be deterministic positive sequences. For all $t\in\N$, assume that almost surely, $A_t$ is positive, and the operators $A_t$, $\A_t$ and $\A$ are invertible (although $\A$ has an   \emph{unbounded} inverse), and that 
$$\|A_t\|\le\amax_t,~\|A_t^{-1}\|\le\amin_t^{-1}.$$

\begin{thm}\label{thm:convergence} Consider the stochastic approximation sequence $(w_t)_{t\in\N_0}$ and remainder $(r_t)_{t\in\N_0}$ defined in (\ref{defwt})-(\ref{defrt}).

Suppose that the Generalized Finiteness Condition holds, and that the variance $\E \|A_t \w_t - b_t \|^2$ is uniformly bounded in $t\in \N$.
Then
\[ \E \|r_t \|^2 \to 0, \]
if the following assumptions hold:

\noindent (A) $\g_t\to0$ and $\DS \sum_{t} \gamma_t \amin_t = \infty$,

\noindent (B) $\DS \limsup_{t\to \infty} \sum_{k=1}^n\g_k^2\prod_{i=k+1}^n(1-\g_i\amin_i)^2 =0$,

\noindent (C)  $\DS \limsup_{t\to \infty} \sum_{k=1}^n\|\Delta_k\|\prod_{i=k+1}^n(1-\g_i\amin_i)=0$.

\end{thm}

The following Lemma \ref{lem:sums} enables us to provide simple sufficient conditions for $(B)$ and $(C)$ in Corollary \ref{cor:convergence}.

\begin{lem}
\label{lem:sums}
Let $(a_t)_{t\in\N}$ and $(b_t)_{t\in\N}$ be two real positive sequences converging to $0$ when $t$ goes to infinity. Then 
$$\limsup_{t\to\iy} a_t/b_t=0\tx{ and }\sum_{t\in\N}b_t=\iy\implies\limsup_{t\to\iy}\sum_{k=1}^n a_k\prod_{i=k+1}^n(1-b_i)=0.$$
\end{lem}

\begin{cor}
\label{cor:convergence}In the statement of Theorem \ref{thm:convergence}, assumptions  $(B)$ and $(C)$ may respectively be replaced by

\noindent (B') $\DS \limsup_{t\to \infty} \frac{\gamma_t}{\amin_t} =0$,

\noindent (C') $\DS \limsup_{t\to \infty} \frac{\|\Delta_t\|}{\amin_t \gamma_t } =0$.
\end{cor}

Theorem \ref{thm:convergence} and Lemma \ref{lem:sums}  are proved in Appendix C, and imply Corollary \ref{cor:convergence}:  Lemma \ref{lem:sums} with $a_t:=\g_t^2$ (resp. $a_t:=\|\Delta_t\|$) and $b_t:=\amin_t \gamma_t$ shows that $(B)$ (resp. $(C)$)  implies $(B')$ (resp. $(C')$).  

The proof of Theorem \ref{thm:convergence} makes use of the following preliminary Lemma \ref{lem:onebnd} (shown in Appendix C), which implies some upper bounds of the norms of operators $\Pi_j^t$, $t\geq j$, also used in Sections \ref{sec:l2} and \ref{sec:hk}.
\begin{lem} \label{lem:onebnd} 
Let  $j_0\in\N$, and let $(\g_t)_{t\in\N}$, $(\amin_t)_{t\in\N}$ and $(\amax_t)_{t\in\N}$be real positive  sequences, and let $(A_t)_{t\in\N}$ be a sequence of positive compact self-adjoint operators on the Hilbert space $\W$. Assume that, for all $t\ge j_0$,  
$\|A_t\|\le\amax_t$ , $\|A_t^{-1}\|\le\amin_t^{-1}$ and $\g_t\amax_t\le1$. 

Then, for all  $t\ge  j_0$ and $j_0\le j\le t$,

(A) $\DS \|I - \gamma_t A_t \|\leq 1 - \g_t\la_t$;

(B) $\DS \|\Pi_j^t\| \leq \prod_{i=j}^t(1-\g_j\la_j)$.

In particular, if the two sequences $(\g_t)_{t\in\N}$ and $(\amin_t)_{t\in\N}$ are such that, for all $t\ge j_0$, $\g_t\la_t:=c\ot^{-1}$ for some $c$, $t_0$ $>0$, then $(B)$ yields $$\|\Pi_j^t\| \leq\left(\frac{j+t_0}{\ot+1}\right)^c.$$
\end{lem}

\subsection{Application to online learning and  Proof of Theorem A}
\label{sec:online}
The online learning sequence $(f_t)_{t\in\N_0}$ defined in (\ref{eq:ft}), with assumptions (I)-(II), can be interpreted as a sequential stochastic approximation algorithm $(w_t)_{t\in\N_0}$ in (\ref{stoch-hilb}), taking values in the Hilbert space $\W:=\H_K$, with 
\begin{eqnarray*}
&A((x,y)):=\langle.,K_x\rangle_K K_x, &b((x,y)):=yK_x\\
& A_t:=A + \lambda_t I, & b_t:=b,
\end{eqnarray*}
so that 
\begin{eqnarray*}
&&\A=L_K,~~\b=L_Kf_\rho,~~\w=f_\rho,\\
&&\A_t=L_K+\la_t I,~~\w_t=f_{\la_t},\\
&&r_t=f_t-f_{\la_t},~~\De_t=f_{\la_t}-f_{\la_{t-1}}.
\end{eqnarray*}

Let us emphasize that the operator $A$ is only defined from $\H_K$ to $\H_K$ here (we would not be able to define $f(x)$ for $f\in\L2$). The properties mentioned below will only hold on $\H_K$ in general, and in particular the norms of operators $\|.\|$ are assumed to be  $\|.\|_{\H_K\to\H_K}$, although operators defined on $\L2$ and commuting with $L_K^{1/2}$ (which is an isometry between $\L2$ and $H_K$) have the same norm in either spaces.
 
Note that $A(z)$ is positive for all $z=(x,y)\in\Z$ (which implies $\A=L_K$ positive as well), since 
$$\langle A((x,y))(f),f\rangle=\langle f(x)K_x,f\rangle=f(x)^2\ge0$$ 
for all $f\in\H_K$. 

Also, for all $f\in\H_K$, $\|Af\|=|\langle K_x,f\rangle |\|K_x\|\leq\|K_x\|^2\|f\|$, so that
$$\|A\|\leq\ka^2,~~\|\A\|\leq\E(\|A\|)\leq\ka^2.$$

Hence
\begin{equation}
\label{atest}
\|A_t\|\leq\amax_t:=\la_t+\ka^2,~~\|A_t^{-1}\|^{-1}\geq\amin_t:=\la_t.
\end{equation}

Let us prove Theorem A; assume its conditions hold. Then the  {\it Generalized Finiteness Condition} of Section \ref{sec:thmconv} is satisfied. Now $f_\rho\in\H_K$ implies $\|f_\la-f_\rho\|_K\to0$ when $\la\to0$.
Therefore the conclusion follows from the convergence of  $\E[\|w_t-\w_t\|^2]=\E[\|f_t-f_{\la_t}\|^2]$ to $0$ in Theorem \ref{thm:convergence}, the condition of uniform boundedness of $\E \|A_t \w_t - b_t \|^2$ being shown in Lemma \ref{lem:xi} (B).

For convenience, we will use, in Sections \ref{sec:hk} and \ref{sec:l2},  the notation
$$L_t:=A(z_t)=\langle. ,K_{x_t}\rangle_K K_{x_t}$$

We will assume that 
\begin{equation}
\label{choice}
\gamma_t = \frac{a}{\ot^{\theta}}, \ \ \ \ \ \lambda_t = \frac{b}{\ot^{1-\theta}}, \ \ \mbox{for some $\theta\in[0,1]$, $\t>0$, $a\in(0,\t^\tet)$, $b\in(0,\t^{1-\tet})$, } 
\end{equation}
and then study the $\H_K$ or $\L2$- norm of the error $f_t-f_\rho$, based using a reverse martingale (resp. martingale) decomposition in Section \ref{sec:hk} (resp.  Section \ref{sec:l2}). We will then optimize the upper bounds in $\theta$, $a$ and $b$ by using some prior information on the regularity of $f_\rho$.

Finally observe that Lemma \ref{lem:onebnd} implies, using (\ref{atest}), that for all $j$, $t$ $\in\N$, $t\ge j$, 
\begin{equation}
\label{bdat}
\|I-\g_tA_t\|\leq1-\g_t\la_t=1-\frac{ab}{\ot},~~~~~\|\Pi_j^t\|\le\left(\frac{j+t_0}{\ot+1}\right)^{ab}
\end{equation}
if $t_0^\tet\geq a(\ka^2+b)$ (and, therefore, $\g_t\amax_t=\g_t\la_t+\g_t\ka^2\le abt_0^{-1}+a\ka^2t_0^{-\tet}\leq1$). 

Similarly, for all $j$, $t$ $\in\N$, $t\ge j$, 
\begin{equation}
\label{bdaht}
\|I-\g_t\A_t\|\leq1-\frac{ab}{\ot},~~~~~~\|\PPi_j^t\|\le\left(\frac{j+t_0}{\ot+1}\right)^{ab}
\end{equation}
if $t_0^\tet\geq a(\ka^2+b)$; the norm in (\ref{bdaht}) can be $\|.\|_{\H_K\to\H_K}$ as well as $\|.\|_{\L2\to\L2}$.
\section{Estimates of Drift on the Regularization Path} \label{sec:drifts}
This section is devoted to estimates on the drift $\|f_\la - f_\mu\|$, $\la,\mu>0$, along the regularization path $\la\to f_{\la}$, in $\H_K$-norm or $\L2$-norm, assuming that $L_K^{-r}f_\rho\in\L2$ for some $r>0$. 
These estimates enable us to upper bound on the one hand the approximation error $\|f_\la-f_\rho\|$ (when specialized to $\mu=0$), 
and on the other hand the {\it drift error} in the martingale and reversed martingale decompositions. 

Note that the estimate $\|f_\la - f_\mu\|_K=O(|\la-\mu|)$ in the case $r=1$ is not improved by increasing $r$.
This is related to a phenomenon usually refered to as the \emph{saturation} problem in regularizations \cite{EngHanNeu96}.

\begin{thm} \label{thm:pie1} Let $\lambda>\mu\geq 0$. Assume that $L_K^{-r}f_\rho\in \L2$ for
some $r\ge-1$.

(A) If $r\in [-1,1]\setminus\{0\}$, then
\[ \|f_\lambda-f_\mu\|_\rho \leq  |\lambda^r - \mu^r| \frac{\|L_K^{-r} f_\rho\|_\rho}{r}; \]

(B) If $r\geq 1$, then for any $1\leq s \leq r$,
\[ \|f_\lambda-f_\mu\|_\rho \leq  \kappa^{2(s-1)} |\lambda - \mu| \|L_K^{-s} f_\rho\|_\rho; \]

(C) If $r\geq 1/2$, then
\[ \|f_\lambda - f_\mu \|_K \leq \frac{|\la - \mu|}{\la} \|f_\rho\|_K; \]

(D) If $r\in [-1/2,3/2]\setminus\{1/2\}$, then
\[ \|f_\lambda-f_\mu\|_K \leq  |\lambda^{r-1/2} - \mu^{r-1/2}| \frac{\|L_K^{-r} f_\rho\|_\rho}{|r-\frac{1}{2}|}; \]

(E) If $r\geq 3/2$, then for any $3/2\leq s \leq r$,
\[ \|f_\lambda-f_\mu\|_K \leq  \kappa^{2(s-3/2)} |\lambda - \mu| \|L_K^{-s} f_\rho\|_\rho. \]
\end{thm}
\begin{proof}
Fix $\lambda>\mu$, assume $L_K^{-r}f_\rho\in \L2$ for
some $r\in[-1,1]$ and let  $\|.\|:=\|.\|_{\L2\to\L2}$. 
We first prove that,  for all $u\ge-1$, if we let 
$$J_{u,\la,\mu}:=(\mu-\lambda)(L_K+\la I)^{-1}(L_K+\mu I)^{-1}L_K^{1+u}$$
then, for all $t\in[-1,1]\setminus\{0\}$, $u\ge t$,
\begin{equation}
\label{ubop}
\|J_{u,\la,\mu}\|\leq\ka^{2(u-t)}|\lambda^t-\mu^t|/|t|.
\end{equation}
	
This will be useful, since
\begin{equation} \label{eq:pierre}
f_\la-f_\mu=(\mu-\lambda)(L_K+\la I)^{-1}(L_K+\mu I)^{-1}L_Kf_\rho
=J_{r,\la,\mu}L_K^{-r}f_\rho,
\end{equation} 
using that 
$$(L_K+\la I)f_\la=L_Kf_\rho,~(L_K+\mu I)f_\mu=L_K f_\rho.$$

Let us prove (\ref{ubop}): using $\|L_K^{u-t}\|=\|L_K\|^{u-t}\le\ka^{2(u-t)}$ by (\ref{ub-lk}), and $\max(t,0)+\min(0,t)=t$, 
\begin{eqnarray*}
\|J_{u,\la,\mu}\|&\leq&|\la-\mu|
\|(L_K+\la I)^{\max(t,0)-1}(L_K+\mu I)^{\min(t,0)}L_K^{-(t+1)}L_K^{1+u}\|\\
&\leq&|\la-\mu|\la^{-1}\la^{\max(t,0)}\mu^{\min(t,0)}\|L_K^{u-t}\|\le\ka^{2(u-t)}|\la-\mu|\la^{-1}\max(\la^{t},\mu^{t})\\
&=&\ka^{2(u-t)}
\Lambda(\mu)|\lambda^t-\mu^t|,
\end{eqnarray*}
where
\[ \Lambda(\mu):=
\begin{cases}
\frac{1-\mu/\lambda}{1-(\mu/\lambda)^t}&\tx{ if }t>0\\
\frac{1-\mu/\lambda}{1-(\lambda/\mu)^t} &\tx{ if }t<0
\end{cases}\]
Now
$$\Lambda(\mu)\le\frac{1}{|t|}.$$
Indeed, if $t>0$, then this is a consequence of 
 $x\leq(1-(1-x)^t)/t$ applied to $x:=1-\mu/\la$, using that  $x\mapsto(1-(1-x)^t)/t$
(defined on $(-\infty,1]$) is convex and thus remains above the tangent
line at $0$. Similarly, we use $x\leq(1-(1-x)^{-t})/(-t)$ if $t<0$.

Now  (\ref{ubop})-(\ref{eq:pierre}) implies (A) with $u:=r$  and $t:=r$, and (B) with $u:=s$ and $t:=1$, since $L_K^{-r}f_\rho\in \L2$ implies $L_K^{-s}f_\rho\in \L2$ for any $s\leq r$. Similarly, (D) (resp. (E)) follows from 
$$L_K^{-1/2}(f_\la-f_\mu)= J_{r-1/2,\la,\mu}L_K^{-r}f_\rho,$$
and (\ref{ubop}) applied to $u:=r-1/2$ and  $t:=u$ (resp. $u:=s-1/2$ and $t:=1$). 

Let us now prove (C): if $r\geq 1/2$, then $f_\rho \in \H_K$, and the first part of equality (\ref{eq:pierre}) implies 
\[ \|f_\la-f_\mu\|_K\leq |\mu-\lambda|\|(L_K+\la I)^{-1}\| \|(L_K+\mu I)^{-1}L_K\| \|L_K^{-1/2}f_\rho\|_\rho \leq \frac{|\mu-\lambda|}{\la}\|f_\rho\|_K. \]
\end{proof}


\section{Upper Bounds for Convergence in $\H_K$} \label{sec:hk}

Throughout this section, we assume
that $L_K^{-r} f_\rho\in \L2$ for some $r\in (1/2,3/2]$, which implies $f_\rho \in \H_K$ with additional regularity, and assume that the sequences $(\g_t)_{t\in\N}$ and $(\la_t)_{t\in\N}$ are chosen in (\ref{choice}).

Our goal is to provide a probabilistic upper bound for
\[ \|f_t - f_\rho \|_K, \]
in order to prove Theorem B. 
We start with the triangle inequality
\[\|f_t - f_\rho\|_K\leq \|f_t - f_{\la_t}\|_K+\|f_{\la_t} - f_\rho\|_K, \]
and apply the reversed martingale decomposition of $(f_t)_{t\in\N}$ developed in Section \ref{sec:general}, Theorem \ref{thm:rmart}:
\begin{equation}  \label{eq:rmart1}
 r_t = \Pi_{1}^t r_0 - \sum_{j=1}^t \gamma_j \Pi_{j+1}^t (A_j \w_j - b_j) - \sum_{j=1}^t \Pi_{j}^t \Delta_j.
\end{equation}
We make use of the corresponding notation of Section \ref{sec:general}, in particular Section \ref{sec:online}, so that 
 $$A_j \w_j - b_j=(L_t+\la_tI)f_{\la_t}-y_tK_{x_t},$$
and
\begin{equation*} \Pi_j^t(x_j,\ldots,x_t)=
\left\{
\begin{array}{lr}
\DS \prod_{i=j}^t \left( I - \gamma_i (L_i + \la_i I) \right) &\tx{ if }j\leq t; \\
I &\tx{ otherwise.}
\end{array}
\right.
\end{equation*}

Now
\[ \|f_t - f_\rho\|_K \leq \Err_{init}(t) + \Err_{samp}(t) + \Err_{drift}(t) + \Err_{approx}(t), \]
where we define the errors as follows:

\noindent (A) \emph{Initial Error}: $\Err_{init}(t)=\|\Pi_1^t r_0\|_K$ comes from the initial choice $f_0$; \\
\noindent (B) \emph{Approximation Error}: $\Err_{approx}(t)=\|f_{\la_t} - f_\rho\|_K$, measures the distance between the regression function and the regularization path at time $t$;\\
\noindent (C) \emph{Drift Error}: $\Err_{drift}(t)= \|\sum_{j=1}^t \Pi_{j}^t \Delta_j\|_K$ comes from the drift along the regularization path $t\mapsto f_{\la_t}$; \\
\noindent (D) \emph{Sample Error}: $\Err_{samp}(t)=\|\sum_{j=1}^t \gamma_j \Pi_{j+1}^t (A_j \w_j - b_j)\|_K$, where $\xi_j =\gamma_j \Pi_{j+1}^t (A_j \w_j - b_j)$
is a reversed martingale difference sequence, reflecting the random fluctuation caused by sampling.

In the remainder of this section, we are going to provide upper bounds for each of the four errors, which, roughly speaking when $ab=1$, are
\begin{eqnarray*}
\Err_{init}(t) & = & O( t^{-1}),  \\
\Err_{approx}(t)  & = & O( t^{-(r-1/2)(1-\theta)} ), \\
 \Err_{drift}(t) & = & O( t^{-(r-1/2)(1-\theta)}), \\
 \Err_{samp}(t) & = & O( t^{\frac{1}{2} - \theta} ).
\end{eqnarray*}

It is not surprising that the approximation error and drift error have the same rate, as both of them come from
the estimates on drifts in Theorem \ref{thm:pie1}. This suggests our explanation that the bias $=\Err_{approx}(t)+ \Err_{drift}(t)$ and the variance $=\Err_{samp}(t)$. Theorem B then follows from these bounds by setting $\theta=2r/(2r+1)$.

\subsection{Initial Error}
\begin{thm}[Initial Error] \label{err:init} Let $\t^\theta\geq a(\ka^2+b)$. Then for all $t\in \N$,
\[ \Err_{init}(t) \leq B_3 \ot^{-ab}, \]
where $B_3 = (\t+1)^{ab}\|r_0\|$.
\end{thm}

\begin{proof}
\[ \Err_{init}(t) \leq \|\Pi_1^t\| \|r_0\| \leq \left(\frac{\t+1}{\ot+1}\right)^{ab} \|r_0\|\leq \left(\frac{\t+1}{\ot}\right)^{ab} \|r_0\| \]
where the second last step uses Lemma \ref{lem:onebnd} (B) with $j=1$.
\end{proof}

\subsection{Approximation Error}

The approximation error is derived from Theorem \ref{thm:pie1}(D) by setting $\la=\la_t$ and $\mu=0$.
\begin{thm}[Approximation Error] \label{err:approx} For $r\in (1/2,3/2]$ and $L_K^{-r}f_\rho \in \L2$,
\[ \|f_{\la_t} - f_\rho \|_K \leq B_1 b^{r-1/2} \ot^{-(r-1/2)(1-\theta)}, \]
where $B_1=(r-1/2)^{-1}\| L^{-r}_K f_\rho \|_\rho$.
\end{thm}

\subsection{Drift Error}

\begin{thm}[Drift Error] \label{err:drift} Let $\t^\theta \geq [a(\kappa^2+b)\vee 1]$. Then for $r\in (1/2,3/2]$ and $L_K^{-r}f_\rho \in \L2$,
\[ \Err_{drift}(t) \leq 
\left\{
\begin{array}{lr}
\displaystyle B_2 b^{r-1/2}\ot^{-(r-1/2)(1-\theta)}, & \mbox{ if } ab>(r-1/2)(1-\tet); \\
\displaystyle B_2 b^{r-1/2}\ot^{-ab}, & \mbox{ if } ab<(r-1/2)(1-\tet),
\end{array}
\right.
\]
where $\DS B_2=\frac{4(1-\theta)}{|ab-(r-1/2)(1-\theta)|}\|L_K^{-r} f_\rho\|_\rho$.
\end{thm}

\begin{proof} We are going to provide an upper bound of
\[ \Err_{drift}(t) = \| \sum_{j=1}^t \Pi_j^t \Delta_j\|_K. \]

First, Lemma \ref{lem:onebnd} implies, using (\ref{atest}), that for all $j$, $t$ $\in\N$, $t\ge j$, 
\begin{equation} \label{eq:pibound}
\|\Pi_{j}^t\|\leq \left(\frac{j+\t}{\ot+1}\right)^{ab}, 
\end{equation}
if $t_0^\tet\geq a(\ka^2+b)$ (and, therefore, $\g_t\amax_t=\g_t\la_t+\g_t\ka^2\le abt_0^{-1}+a\ka^2t_0^{-\tet}\leq1$). 

Second, by Theorem \ref{thm:pie1}(D), 
\begin{eqnarray} 
\|\Delta_t\|_K & = & \|f_{\lambda_t}-f_{\lambda_{t-1}}\|_K  \leq  \left|\la_t^{r-1/2}-\la_{t-1}^{r-1/2} \right| \frac{\|L_K^{-r} f_\rho\|_\rho}{r-\frac{1}{2}} \nonumber \\
&\leq & b^{r-1/2} (1-\theta)  (\ot-1)^{-(r-1/2)(1-\theta)-1} \|L_K^{-r} f_\rho\|_\rho, \label{eq:meanvaluethm}
\end{eqnarray}
where we use
\begin{eqnarray}
|\lambda_{t}^{r-1/2} - \lambda_{t-1}^{r-1/2} | & = & b^{r-1/2} \left| \ot^{-(r-1/2)(1-\theta)} - (\ot-1)^{-(r-1/2)(1-\theta)}\right| \nonumber \\
& \leq &  b^{r-1/2} (r-1/2)(1-\theta) (\ot-1)^{-(r-1/2)(1-\theta)-1},   \nonumber 
\end{eqnarray}
due to the Mean Value Theorem with $h(x)=x^{-(r-1/2)(1-\theta)}$ and $h^\prime(x) = -(r-1/2)(1-\theta) x^{-(r-1/2)(1-\theta)-1}$, such that
\[ |h\ot-h(\ot-1)|=|h^\prime(\eta)|\leq |h^\prime(\ot-1)|, \ \ \ \ \ \mbox{for some $\eta\in(\ot-1,\ot)$}. \]

Now combining (\ref{eq:pibound}) and (\ref{eq:meanvaluethm}) gives 
\begin{eqnarray*}
\Err_{drift}(t) &=& \| \sum_{j=1}^t \Pi_j^t \Delta_j\|_K  \leq b^{r-1/2} (1-\tet) \|L_K^{-r} f_\rho\|_\rho \cdot \sum_{j=1}^t \left(\frac{j+\t}{\ot+1}\right)^{ab} (j+\t -1)^{-(r-1/2)(1-\tet)-1} \\
& \leq & \frac{4b^{r-1/2}(1-\theta)\|L_K^{-r} f_\rho\|_\rho}{(\ot+1)^{ab}}\sum_{j=1}^t (j+\t)^{ab-1-(r-1/2)(1-\theta)}.  
\end{eqnarray*}
It suffices to bound 
\[ \sum_{j=1}^t (j+\t)^{ab-1-(r-1/2)(1-\theta)} \leq \int_0^{t+1} (x+\t)^{ab-1-(r-1/2)(1-\theta)} dx  =: I_t\]
Now, if $ab>(r-1/2)(1-\tet)$,  then
\[ 
I_t  \leq \frac{(\ot+1)^{ab-(r-1/2)(1-\theta)}}{ab-(r-1/2)(1-\theta)};
\]
whereas $ab<(r-1/2)(1-\tet)$ implies
\[ 
I_t \leq \frac{\t^{ab-(r-1/2)(1-\tet)}}{(r-1/2)(1-\theta)-ab}  \leq \frac{1}{|ab-(r-1/2)(1-\theta)|}, \ \ \ \t\geq 1.
\]
\end{proof}

\subsection{Sample Error}

\begin{thm} [Sample Error] \label{err:samp}
Assume that $\t^\theta\geq [a(\kappa^2+b)\vee b \vee 1]$ and $ab \neq \tet-1/2$ or $(3\tet-1)/2$. Then,  with probability
at least $1-\delta$ ($\delta\in (0,1)$),
\[ \Err_{samp}(t) \leq B_4  ab^{-1/2}\ot^{-\left[ab \wedge \frac{3\tet-1}{2}\right]} + B_5 a\ot^{-[ab \wedge (\theta-1/2)]}  \]
where $\DS B_4=\frac{2(\ka+1)^2 \M}{3} \log\frac{2}{\delta} $ and $\DS B_5 = \frac{8\ka \M}{\sqrt{|ab-(\tet-1/2)|}}\log\frac{2}{\delta}$ .
\end{thm}

The proof of Theorem \ref{err:samp} requires some auxilary estimates.

\begin{lem} \label{lem:xi} Let $A_t \w_t - b_t=(f_{\lambda_t}(x_t) - y_t) K_{x_t} + \lambda_t f_{\lambda_t}$.

(A) $\|A_t \w_t - b_t\|_K\leq (\ka +1)^2 M_\rho/\sqrt{\lambda_t}$, if $\t^{1-\tet}\ge b$;

(B) $\E[\|A_t \w_t - b_t\|_K^2] \leq 4 \kappa^2 M_\rho^2$.
\end{lem}

\begin{proof}
(A) Using $\|f_\lambda\|_K \leq M_\rho/\sqrt{\lambda}$ in Lemma \ref{lem:fla}(A),
\[ \|A_t \w_t - b_t\| \leq \|f_{\lambda_t}(x_t) K_{x_t}\|_K + |y_t| \| K_{x_t}\|_K + \lambda_t \|f_{\la_t}\|_K \leq M_\rho \kappa^2 /\sqrt{\lambda_t} + M_\rho \kappa + M_\rho \sqrt{\lambda_t} \]
since $\|f_{\lambda_t}(x_t) K_{x_t}\|_K = |\<f_{\lambda_t}, K_{x_t}\>|\| K_{x_t}\|_K  \leq \|f_{\lambda_t}\|_K \|K_{x_t}\|_K^2
\leq M_\rho \kappa^2 /\sqrt{\lambda_t}$.
Now,
\begin{eqnarray*}
M_\rho \kappa^2 /\sqrt{\lambda_t} + M_\rho \kappa + M_\rho \sqrt{\lambda_t} & \leq & (\kappa^2+\ka +1) M_\rho/\sqrt{\lambda_t} \leq (\ka +1)^2 M_\rho/\sqrt{\lambda_t}
\end{eqnarray*}
where the second last inequality is due to $\t^{1-\tet} \ge b \Rightarrow \la_t \leq 1$. 

(B) Using $\lambda_t f_{\la} = L_K f_\rho - L_K f_{\la}$ we obtain
\[ (f_{\lambda_t}(x_t) - y_t) K_{x_t} + \lambda_t f_{\la_t} = (L_t - L_K)f_{\la_t} +L_K f_\rho- y_t K_{x_t}.   \]
\begin{eqnarray*}
\E[\|A_t \w_t - b_t\|^2] & = &\E \|(L_t - L_K)f_{\la_t} +L_K f_\rho- y_t K_{x_t}\|_K^2 \\
& \leq & 2\E[\|(L_t - L_K)f_{\la_t}\|_K^2 + \|L_K f_\rho- y_t K_{x_t}\|_K^2] \\
& \leq & 2\E[\|L_t f_{\la_t}\|_K^2 +  \|y_t K_{x_t}\|_K^2] \leq 2 \kappa^2 (\|f_{\la_t}\|_\rho^2 + M^2_\rho) = 4 \kappa^2 M^2_\rho
\end{eqnarray*}
since $\E[L_t] = L_K$, $\E[y_t K_{x_t}] = L_K f_\rho$ and $\|f_\la\|_\rho \leq \M$ by Lemma \ref{lem:fla}(B).
\end{proof}

Now we are ready to give the proof of the sample error bounds, Theorem \ref{err:samp}.

\begin{proof}[Proof of Theorem \ref{err:samp}] We are going to bound 
\[ \Err_{samp}(t) = \left\| \sum_{j=1}^t \xi_j \right\|_K \]
where $\xi_j = \ga_j\Pi_{j+1}^t ( A_j \w_j - b_j)$ is a reversed martingale difference sequence. To apply the Pinelis-Bernstein inequality in Proposition \ref{prop:bernstein}, we need bounds on $\|\xi_j\|_K$ and $\E_{j+1} \|\xi_j\|_K^2$ where $\E_{j+1}[\cdot]$ is the expectation conditional on examples after time $j$.

Notice that for $\t^\tet \geq a(\ka^2+b)$ and $j\geq 1$, using $1+x\le e^x$ for all $x\in\R$, 
\[ \|\ga_j \Pi_{j+1}^t \| \leq \frac{a}{(j+\t)^\theta} \left(\frac{j+\t+1}{\ot+1}\right)^{ab}\leq \frac{a (j+\t)^{ab-\theta}}{(\ot+1)^{ab}}
(1+t_0^{-1})^{ab}\leq\frac{ea (j+\t)^{ab-\theta}}{(\ot+1)^{ab}},\]
where $e$ is the Euler constant.

Now Lemma \ref{lem:xi} (B) implies
\[  \E[\|A_t \w_t - b_t\|_K^2] \leq 4 \kappa^2 \M^2. \]
Hence 
\[
\E_{j+1}\|\xi_j\|^2  \leq  \frac{4e^2 (a\ka \M)^2 (j+\t)^{2ab-2\theta}}{(\ot+1)^{2ab}} ,
\]
so that, if $t_0\ge2$,  
\begin{eqnarray}
\sum_{j=1}^t \E_{j+1}\|\xi_j\|^2 & \leq & \left\{ 
\begin{array}{lr}
\DS  \frac{2e^2 (a\ka \M)^2 }{ab-(\tet-1/2)}  (\ot+1)^{-2(\tet-1/2)}, & \mbox{ if $ab>\tet-1/2$} \\
\DS \frac{2e^2 (a\ka \M)^2 }{(\tet-1/2)-ab}  (\ot+1)^{-2ab}, &  \mbox{ if $ab<\tet-1/2$} 
\end{array}
\right. \label{eq:sigmaBound}
\end{eqnarray}

On the other hand, if $\t^{1-\tet} \ge b$, Lemma \ref{lem:xi} (A) implies
 \[ \| A_j \w_j - b_j\|_K \leq (\ka+1)^2 \M /\sqrt{\la_j}, \]
 whence 
\begin{eqnarray}
\|\xi_j\|_K & \leq & \frac{ea(\ka+1)^2 \M}{\sqrt{b}}\cdot \frac{(j+\t)^{ab-(3\tet-1)/2}}{(\ot+1)^{ab}} \nonumber \\
& \leq & 
\left\{ 
\begin{array}{lr}
\DS  \frac{ea(\ka+1)^2 \M}{\sqrt{b}} \ot^{-(3\tet-1)/2} , & \mbox{ if $ab\geq (3\tet-1)/2$} \\
\DS \frac{ea(\ka+1)^2 \M}{\sqrt{b}} \ot^{-ab}, &  \mbox{ if $ab\leq (3\tet-1)/2$} 
\end{array}
\right. \label{eq:MBound}
\end{eqnarray}  
The final bound is obtained by Pinelis-Bernstein inequality in Proposition \ref{prop:bernstein} with upper bounds (\ref{eq:sigmaBound}) and (\ref{eq:MBound}). 
\end{proof}

\subsection{Proof of Theorem B}
We choose $\tet=2r/(2r+1)$, $a\ge1$, $b\le1$ such that $ab=1$, and assume $t_0^\tet\ge a\ka^2+1$. Using Theorems \ref{err:approx}, \ref{err:drift}, \ref{err:init}, and \ref{err:samp},
\beq
\|f_t - f_\rho\|_K &\leq &\Err_{init}(t)+\Err_{approx}(t) + \Err_{drift}(t)+\Err_{samp}(t) \\
&\leq & B_3\ot^{-ab}+[(B_1+B_2)a^{1/2-r}+(B_4\t^{-\tet}\sqrt{a}+B_5)a]\ot^{-(2r-1)/(4r+2)}
\eeq

Note that, by Lemma \ref{lem:xi}(A) with $f_0=0$, 
 \[ B_3 = (\t+1) \|r_0\|=(\t+1)\|f_{\la_0}\|\leq C_0:=2\t \frac{M_\rho}{\sqrt{\la_0}} =2 \t^{\frac{4r+3}{4r+2}}  M_\rho \]
On the other hand, 
$$C_1:=B_1+B_2 =  \left(\frac{2}{2r-1}+\frac{8}{2r+3}  \right)\| L^{-r}_K f_\rho \|_\rho=\frac{20r-2}{(2r-1)(2r+3)}\| L^{-r}_K f_\rho \|_\rho$$
and, using  $\sqrt{a}t_0^{-\tet}\le\ka^{-1}$ and $t_0\ge1$, 
$$B_4\t^{-\tet}\sqrt{a}+B_5\le\frac{2(\ka+1)^2M_\rho}{3\ka}+\frac{8\ka M_\rho}{\sqrt{3/4}}
\le C_2:=\frac{20 (\ka+1)^2M_\rho}{\ka},$$
which concludes the proof of Theorem B.

\section{Upper Bounds for Convergence in $\L2$} \label{sec:l2}


Throughout this section, we assume  
that $L_K^{-r} f_\rho\in \L2$ for some $r\in [1/2,3/2]$, which  implies $f_\rho \in \H_K$ with additional regularity, and assume the  sequences $(\g_t)_{t\in\N}$ and $(\la_t)_{t\in\N}$ are chosen in (\ref{choice}). Note that the case $r=1/2$ is included here, whereas it was not in Section \ref{sec:hk} and Theorem B.

Our goal is to provide a probabilistic upper bound of 
\[ \|f_t - f_\rho \|_\rho, \]
in  order to prove Theorem C.  
As in Section \ref{sec:hk}, we start with  the triangle inequality
\[\|f_t - f_\rho\|_\rho \leq \|f_t - f_{\la_t}\|_\rho+\|f_{\la_t} - f_\rho\|_\rho, \]
but apply here the martingale decomposition of  $(f_t)_{t\in\N_0}$ in $\L2$ developed in  Theorem \ref{thm:mart} instead:
\[ r_{t} = \PPi_1^t r_0 + \sum_{j=1}^t \gamma_j \PPi_{j+1}^t \chi_j - \sum_{j=1}^t \PPi_{j}^t \Delta_j. \]
We make use of the corresponding notation of Section \ref{sec:general}, in particular \ref{sec:online}, so that 
$$\chi_t=( L_K-L_t ) f_{t-1} +( y_t K_{x_t}-L_K f_\rho ),$$ 
and
\begin{equation}
\PPi_j^t=
\left\{
\begin{array}{lr}
\DS \prod_{i=j}^t \left( I - \gamma_i  (L_K + \la_i I)  \right) &\tx{ if }j\leq t; \\
I  &\tx{ otherwise.}
\end{array}
\right.
\end{equation}

The martingale decomposition enables us to make use of the isometry $L_K^{1/2}:\L2 \to \H_K$, in the sense that
one can benefit from the spectral decomposition of $L_K^{1/2}\PPi_j^t$ to get a tighter estimate. This was not possible with the reversed martingale decomposition, since  $L_K^{1/2} \Pi_j^t$ does not have an obvious spectral decomposition.

Note however that $\chi_t$ depends on $f_{t-1}$, so that we need preliminary estimates of $\|\chi_t\|_\rho$, provided in Appendix B. 

As in Section \ref{sec:hk}, we introduce the following definitions for convenience.

{\bf [Definitions of Errors]}

\noindent (A) \emph{Initial Error}: $\Err_{init}(t)=\|\PPi_1^t r_0\|_\rho $, which reflects the propagation error by the initial choice $f_0$; \\
\noindent (B) \emph{Approximation Error}: $\Err_{approx}(t)=\|f_{\la_t} - f_\rho\|_\rho$, which measures the distance between the regression function and
the regularization path at time $t$;\\
\noindent (C) \emph{Drift Error}: $\Err_{drift}(t)= \|\sum_{j=1}^t \PPi_{j}^t \Delta_j\|_\rho $, which measures the error caused by drifts from
$f_{\la_{j-1}}$ to $f_{\la_j}$ along the regularization path; \\
\noindent (D) \emph{Sample Error}: $\Err_{samp}(t)=\|\sum_{j=1}^t \gamma_j \PPi_{j+1}^t \chi_j \|_\rho $, where $\chi_j$
is a martingale difference sequence, reflecting the random fluctuation caused by sampling.

Our aim is to  bound
\[ \|f_t - f_\rho\|_\rho \leq \Err_{init}(t) + \Err_{samp}(t) + \Err_{drift}(t) + \Err_{approx}(t). \]
In the remainder of this section, we are going to provide upper bounds for each of the four errors, which, roughly speaking when $ab=1$, are
\begin{eqnarray*}
\Err_{init}(t) & = & O( t^{-1})  \\
\Err_{approx}(t)  & = & O( t^{-r(1-\theta)} ) \\
 \Err_{drift}(t) & = & O( t^{-r(1-\theta)}) \\
  \Err_{samp}(t) & = & O( t^{-\theta/2} )
\end{eqnarray*}

This suggests our explanation that the bias $=\Err_{approx}(t)+ \Err_{drift}(t)=O(t^{-r(1-\theta)})$ and the variance $=\Err_{samp}(t)=O(t^{-\theta/2})$ similar to the batch learning setting. Theorem C then follows from these bounds by setting $\theta=2r/(2r+1)$.

\subsection{Initial Error}

\begin{thm}[Initial Error]\label{err1:init} Let $\t^\theta \geq a(\ka^2+b)$. Then for all $t\in \N$,
\[ \Err_{init}(t) \leq B_6 \ot^{-ab}, \]
where $B_6 = \M (\t+1)^{ab}$.
\end{thm}
\begin{proof}
Lemma \ref{lem:onebnd}(B) with $j=1$ and \eqref{atest} imply that , if $\t^\theta \geq a(\ka^2+b)$, 
\[ \Err_{init}(t) \leq \|\PPi_1^t\|\|r_0\| \leq \left(\frac{\t+1}{\ot+1} \right)^{ab}\|r_0\|. \]
For $f_0=0$, using Lemma \ref{lem:fla}(B), $\|r_0 \|_\rho = \|f_{\la_0}\|_\rho \leq \M$.
\end{proof}

\subsection{Approximation Error}

\begin{thm}[Approximation Error]\label{err1:approx} For $r\in (0,1]$ and $L_K^{-r}f_\rho \in \L2$,
\[ \Err_{approx}(t)\leq B_7 b^r \ot^{-r(1-\theta)}, \]
where $B_7=r^{-1}\| L^{-r}_K f_\rho \|_\rho$.
\end{thm}
\begin{proof}
Follows from Theorem \ref{thm:pie1}(A) with $\la=\la_t$ and $\mu=0$.
\end{proof}

\subsection{Drift Error}

\begin{thm}[Drift Error]\label{err1:drift}
Assume $\t^\theta \geq [a(\kappa^2+b)\vee 1]$. Then, if $r\in (0,1]$ and $L_K^{-r}f_\rho \in \L2$,
\[ \Err_{drift}(t) \leq 
\left\{
\begin{array}{lr}
\displaystyle B_8 b^{r}\ot^{-r(1-\theta)}, & \mbox{ if } ab>r(1-\tet); \\
\displaystyle B_8 b^{r}\ot^{-ab}, & \mbox{ if } ab<r(1-\tet),
\end{array}
\right.
\]
where $\DS B_8=\frac{4(1-\theta)}{|ab-r(1-\theta)|}\|L_K^{-r} f_\rho\|_\rho$.
\end{thm}

\begin{proof} Similar to the proof of Theorem \ref{err:drift}, replacing $r-1/2$ with $r$. 
\end{proof}

\subsection{Sample Error}
In this section we assume $b=a^{-1}$ for simplicity; this is necessary for the bounds in Appendix B, in particular Corollary \ref{loght}, and is enough to provide the optimal bounds we need (see discussion after statement of Theorem A).
 
\begin{thm}[Sample Error]\label{err1:samp} Assume that $L_K^{-r} f_\rho\in \L2$ for some $r\in [1/2,1]$, $\tet\in[1/2,2/3]$, $ab=1$, $a\ge4$ and $t_0^\tet\ge 2+8\kappa^2a$. 
Then, for all $t\in \N$,  with probability at least $1-\delta$ ($\delta\in (0,1)$)
\[ \Err_{samp}(t)\leq \sqrt{a}B_{9}\frac{\log (2/\de)}{\ot^{\tet/2}}+\left(a^{3/2}B_{10}\sqrt{\log \ot}+a^{5/2}B_{11}\right)\frac{(\log (2/\de))^2}{\ot^{(3\tet-1)/2}},\]
where
$$B_9:=10\ka\M,\,\,\,B_{10}:=63\ka^2\M,\,\,\,\,B_{11}:=50\ka^2\M\t^{1/2-\tet}.$$

\end{thm}
\bp
Fix $t\in\N$, $\de\in[0,1]$, and let
\begin{equation} \label{eq:Atd}
A_{t,\de}:=\ka M_\rho a\left[12at_0^{1/2-\tet}+15\sqrt{\log \ot}\right]\log\frac{2}{\de}.
\end{equation}

For all $j\in\N$, let us define the martingale difference sequence
$$X_j := \gamma_j \PPi_{j+1}^t \chi_j\1_{\{\|h_{j-1}\|_K (j+\t)^{\tet-1/2} \le A_{t,\de}\}},$$ 
where we make use of the notation of Appendix B. Recall that Corollary \ref{loght} implies, with probability at least $1-\delta$, all the indicator function events for $1\leq j <t$  hold, which will be assumed in the computation below.

Recall that
\[ \chi_j = (\A_j - A_j)w_{j-1} + b_j - \b_j = (L_K - L_j)f_{j-1} + y_j K_{x_j} - L_K f_\rho \]
where $L_j := \<\ , K_{x_j}\>K_{x_j} $.

Using Lemma \ref{ubgt} and the decomposition $f_j=f_\rho+g_j+h_j$ in Appendix B, we deduce that, for all $1\le j< t$, if 
$\|h_{j-1}\|_K (j+\t)^{\tet-1/2} \le A_{t,\de}$, 
\begin{eqnarray*}
\E_{j-1}\|\chi_j\|_K^2 &=& \E_{j-1}\|y_j K_{x_j}- L_j f_{j-1}-(L_K f_\rho-L_K f_{j-1}) \|_K^2 \leq \E_{j-1}\|y_j K_{x_j}- L_j f_{j-1}\|_K^2 \\
&\le&\E_{j-1}\|(y_j K_{x_j}- L_j f_\rho)-L_jg_{j-1}-L_j h_{j-1}\|_K^2\\
&\le&3\ka^2 [\E_{j-1}|y_j-f_\rho(x_j)|^2+\E_{j-1}|g_{j-1}(x_j)|^2+\E_{j-1}|h_{j-1}(x_j)|^2]\\
&\le&3\ka^2[4\M^2+\|g_{j-1}\|_\rho^2+\|h_{j-1}\|_\rho^2]\le3\ka^2[5\M^2+\ka^2(j+t_0)^{1-2\tet}A_{t,\de}^2]=:A'_{j,t,\de}
\end{eqnarray*}
Now, using the isometry $L_K^{1/2}:\L2\to \H_K$,
\begin{eqnarray*}
\sum_{j=1}^t\E_{j-1}\|X_j\|_\rho^2&=&\sum_{j=1}^t\E_{j-1}\|L_K^{1/2}X_j\|_K^2
= \sum_{j=1}^t\g_j^2\E_{j-1}\|L_K^{1/2} \PPi_{j+1}^t \chi_j\|_K^2\\
&\le&\sum_{j=1}^t\left(\g_j^2\|\PPi_{j+1}^tL_K\PPi_{j+1}^t\|\right)\E_{j-1}\|\chi_j\|_K^2
\le\sum_{j=1}^t\g_j^2A'_{j,t,\de}\|\PPi_{j+1}^tL_K\PPi_{j+1}^t\|
\end{eqnarray*}
In order to estimate $\sum_{j=1}^t \gamma_j^2 A'_{j,t,\de} \| \PPi_{j+1}^t L_K \PPi_{j+1}^t\|$, recall that $(\mu_\al, \phi_\al)_{\al\in \N}$ is an orthonormal eigen-system of $L_K:\L2\to \L2$.
Let $a_i = \ga_i \la_i + \ga_i \mu_\al$ for simplicity; then
\begin{eqnarray*}
&& \sum_{j=1}^t \gamma_j^2A'_{j,t,\de} \| \PPi_{j+1}^t L_K \PPi_{j+1}^t\| \leq \sup_{\mu_\al} \sum_{j=1}^t \ga_j^2A'_{j,t,\de} \mu_\al \prod_{i=j+1}^t (1-a_i)^2 \\
& = & \sup_{\mu_\al} \sum_{j=1}^t \left [\ga_jA'_{j,t,\de}  \prod_{i=j+1}^t (1-a_i) \right]  \cdot \left[\ga_j \mu_\al \prod_{i=j+1}^t (1-a_i)\right] \\
& \leq & \sup_{\mu_\al} \left\{\left [ \sup_{j}\ga_jA'_{j,t,\de}  \prod_{i=j+1}^t (1-a_i) \right] \cdot
\left[ \sum_{j=1}^t \ga_j \mu_\al \prod_{i=j+1}^t (1-a_i)\right] \right\}
\end{eqnarray*}
where for large enough $\t$,
\begin{eqnarray*}  \label{eq:sup}
&&\sup_j \ga_j A'_{j,t,\de} \prod_{i=j+1}^t (1-a_i)\leq \sup_j \ga_j A'_{j,t,\de} \prod_{i=j+1}^t (1- \ga_i \la_i) \\
&\leq& 3a \ka^2\sup_j \frac{j+\t}{\ot} \cdot\left(\frac{5\M^2}{(j+\t)^\theta}+\frac{\ka^2 A^2_{t,\de}}{(j+\t)^{3\tet-1}}\right) \\
&\leq& 3a \ka^2 \left(\frac{5\M^2}{\ot^\theta}+\frac{\ka^2 A^2_{t,\de}}{\ot^{3\tet-1}}\right),  
\end{eqnarray*}
and
\[ \sum_{j=1}^t \ga_j \mu_\al \prod_{i=j+1}^t (1-a_i) \leq  \sum_{j=1}^t (1 - (1-\ga_j \mu_\al))\prod_{i=j+1}^t (1-\ga_i \mu_\al ) =
1 - \prod_{i=1}^t(1-\ga_i \mu_\al ) \leq 1.\]
These two upper bounds give
\begin{equation} \label{ubvar}
\sum_{j=1}^t\E_{j-1}\|X_j\|_\rho^2 \leq \frac{3a\ka^2}{\ot^{\theta}}\left(5\M^2 + \frac{\ka^2 A^2_{t,\de}}{\ot^{2\tet-1}}\right) =:\sigma_t^2.
\end{equation}

Moreover, again  if 
$\|h_{j-1}\|_K (j+\t)^{\tet-1/2} \le A_{t,\de}$, then, using Lemma \ref{ubgt} (B) and Corollary \ref{loght}, we deduce
\begin{eqnarray*}
\|y_j K_{x_j}- L_j f_{j-1} \|_K  & = & \| y_j K_{x_j} - L_j (f_\rho + g_{j-1} + h_{j-1}) \|_K \\
& \le & \ka M_\rho +\frac{\ka^2M_\rho}{\sqrt{\la_j}}+\ka^2A_{t,\de}(j+t_0)^{1/2-\tet}=: C_{j,t,\delta},
\end{eqnarray*}
which implies
$$\|\chi_j\|_K=\|y_j K_{x_j}- L_j f_{j-1}-\E_j[y_j K_{x_j}- L_j f_{j-1}]\|_K\le 2C_{j,t,\delta}.$$

Therefore
\begin{eqnarray*}
\|X_j\|_\rho & \leq & \g_j \|L_K^{1/2} \PPi_{j+1}^t \chi_j\|_K \leq 2\g_j C_{j,t,\delta} \| \PPi_{j+1}^t  L_K \PPi_{j+1}^t\|_K^{1/2} \\
& \le &  2\ka \sup_j \ga_j C_{j,t,\de} \prod_{i=j+1}^t (1- \ga_i \la_i), \ \ \  \|L_K\| \leq \ka^2 \\
&\le &2 a\ka^2 \sup_j \frac{j+\t}{\ot} \cdot\left(\frac{\M}{(j+t_0)^\tet}+\frac{\ka\M\sqrt{a}}{(j+\t)^{(3\tet-1)/2}}+\frac{\ka A_{t,\de}}{(j+\t)^{2\tet-1/2}}\right) \\
&\le &2 a\ka^2\left(\frac{\M}{\ot^\tet}+\frac{\ka\M\sqrt{a}}{\ot^{(3\tet-1)/2}}+\frac{\ka A_{t,\de}}{\ot^{2\tet-1/2}}\right)\\
&\le &\frac{2\sqrt{a}\ka}{\ot^{\tet/2}}\left(\M+\ka\frac{\ka\M a+A_{t,\de}}{\ot^{\tet-1/2}}\right)=:M,
\end{eqnarray*}
where we use $\t\ge\ka^2 $ twice in the last inequality.

Combining $M$ and $\s_t$ from  (\ref{ubvar}), we obtain
\begin{eqnarray*}
2\left(\frac{M}{3}+\s_t\right)&=&\frac{2\sqrt{a}\ka}{\ot^{\tet/2}}
\left[\left(\sqrt{15}+\frac{2}{3}\right)\M+\ka\frac{(\sqrt{3}+1/3)A_{t,\de}+\ka\M a/3}{\ot^{\tet-1/2}}\right]\\
&\le&\frac{\sqrt{a}B_9}{\ot^{\tet/2}}+\left(a^{3/2}B_{10}\sqrt{\log \ot}+a^{5/2}B_{11}\right)\frac{\log (2/\de)}{\ot^{(3\tet-1)/2}},
\end{eqnarray*}
where we use that
$$B_9=10\ka\M\ge 2\ka(\sqrt{15}+2/3)\M,$$
$$B_{10}=63\ka^2\M\ge\ka^2\M[30(\sqrt{3}+1/3)+2/(3\log 2)],$$
$$B_{11}=50\ka^2\M\t^{1/2-\tet}\ge24\ka^2\M\t^{1/2-\tet}(\sqrt{3}+1/3).$$
\end{proof}

\subsection{Proof of Theorem C}
We choose $\tet=2r/(2r+1)$, $a\ge1$, $b\le1$ such that $ab=1$, and assume $t_0^\tet\ge a\ka^2+1$. 
Using Theorems, \ref{err1:approx}, \ref{err1:drift}, \ref{err1:init}, and \ref{err1:samp}
\beq
\|f_t - f_\rho \|_\rho &\leq & \Err_{init}(t) +\Err_{approx}(t) + \Err_{drift}(t) + \Err_{samp}(t)\\
&\leq&\frac{B_6}{\ot}+\left((B_7+B_8)a^{-r}+\sqrt{a}B_9\log \frac{2}{\de}\right)\left(\frac{1}{\ot}\right)^{\frac{r}{2r+1}}+\left(a^{3/2}B_{10}\sqrt{\log \ot}+a^{5/2}B_{11}\right)\frac{\left(\log (2/\de\right)^2}{\ot^{\frac{6r-1}{4r+2}}},
\eeq
This enables us to conclude, with $D_0:=2\M\t\ge B_6=\M(\t+1)$, 
$$D_1:= B_7+B_8=\frac{5r+1}{r(1+r)}\|L_K^{-r}f_\rho\|_\rho,$$
$D_2:=B_9$, $D_3:=B_{10}$, and $D_4:=B_{11}$.
\section*{Appendix A: A Probabilistic Inequality}
\renewcommand{\thesection}{A}
\setcounter{equation}{0} \setcounter{thm}{0}
\renewcommand{\thethm}{A.\arabic{thm}}
\renewcommand{\theequation}{A-\arabic{equation}}

The following result is quoted from [Theorem 3.4 in \citeNP{Pinelis94}].

\begin{lem}[Pinelis-Bennett]
Let $\xi_i$ be a martingale difference sequence in a Hilbert space. Suppose that almost surely $\|\xi_i\|\leq M$ and 
$\sum_{i=1}^t \E_{i-1} \|\xi_i\|^2 \leq \sigma^2_t$.
Then
\[ \Prob \left\{\sup_{1\le k\le t} \left\|\sum_{i=1}^k \xi_i \right\| \geq \epsilon \right\} \leq
2 \exp \left\{-\frac{\sigma^2_t}{M^2} g \left(\frac{M\epsilon}{\sigma^2_t} \right) \right\}, \]
where $g(x) = (1+x)\log(1+x)-x$ for $x>0$.
\end{lem}

Using the lower bound $g(x)\geq \frac{x^2}{2(1+x/3)}$, one may obtain the following generalized Bernstein's inequality.

\begin{cor}[Pinelis-Bernstein]
Let $\xi_i$ be a martingale difference sequence in a Hilbert space. Suppose that almost surely $\|\xi_i\|\leq M$ and 
$\sum_{i=1}^t \E_{i-1} \|\xi_i\|^2 \leq \sigma^2_t$. Then
\begin{equation} \label{eq:bernstein}
\Prob \left\{ \sup_{1\le k\le t}\left\|\sum_{i=1}^k \xi_i \right\| \geq \epsilon \right\} \leq
2 \exp \left\{-\frac{\epsilon^2}{2(\sigma_t^2+M\epsilon/3)} \right\}.
\end{equation}
\end{cor}

The following result will be used as a basic probabilistic inequality to derive various bounds.

\begin{prop} \label{prop:bernstein}
Let $\xi_i$ be a martingale difference sequence in a Hilbert space. Suppose that almost surely $\|\xi_i\|\leq M$ and 
$\sum_{i=1}^t \E_{i-1} \|\xi_i\|^2 \leq \sigma^2_t$.
Then the following holds with probability at least $1-\delta$ ($\delta\in (0,1)$),
\[ \sup_{1\le k\le t}\left\|\sum_{i=1}^k \xi_i \right\| \leq 2\left(\frac{M}{3} + \sigma_t\right) \log\frac{2}{\delta} .   \]
\end{prop}

\begin{proof}
Taking the right hand side of (\ref{eq:bernstein}) to be $\delta$, then we arrive at the following quadratic equation for $\epsilon$,
\[ \epsilon^2 - \frac{2 M }{3}\epsilon \log\frac{2}{\delta} - 2 \sigma_t^2 \log\frac{2}{\delta} = 0. \]
Note that $\epsilon>0$, then
\begin{eqnarray*}
\epsilon & = & \frac{1}{2} \left\{ \frac{2M}{3}   \log\frac{2}{\delta} + \sqrt{\frac{4M^2}{9}   \log^2\frac{2}{\delta} + 8 \sigma^2_t \log\frac{2}{\delta}} \right\} \\
& = & \frac{M}{3}   \log\frac{2}{\delta} + \sqrt{\left(\frac{M}{3}\right)^2   \log^2\frac{2}{\delta} + 2 \sigma^2_t \log\frac{2}{\delta}} \\
& \leq & \frac{2M}{3}   \log\frac{2}{\delta}+ \sqrt{ 2 \sigma^2_t \log\frac{2}{\delta}},
\end{eqnarray*}
where the second last step is due to $\sqrt{a^2 + b^2}\leq a + b$ ($a,b>0$) with
$$a= \frac{M}{3}\log\frac{2}{\delta},\ \ \ \mbox{and} \ \ \  b=\sqrt{2\sigma_t^2 \log\frac{2}{\delta}}. $$
We complete the proof by relaxing $\sqrt{2\sigma_t^2 \log 2/\delta} \leq 2 \sigma_t \log 2/\delta$ since $2\log 2/\delta>1$ for $\delta \in (0,1)$.
\end{proof}


\section*{Appendix B: Preliminary Upper Bounds }
\renewcommand{\thesection}{B}
\setcounter{equation}{0} \setcounter{thm}{0}
\renewcommand{\thethm}{B.\arabic{thm}}
\renewcommand{\theequation}{B-\arabic{equation}}
Appendix B is devoted to the proof of preliminary upper bounds on the online learning sequence $(f_t)_{t\in\N}$ defined in (\ref{eq:ft}), and on the regularization path $\la\mapsto f_{\la}$. We make use of the notation of Section \ref{sec:general}, in particular Section \ref{sec:online}. For simplicity  we assume $f_0:=0$; note that another choice would correspond to adding $\Pi_1^tf_0$ to $f_t$ at time $t$. We assume that the sequences $(\g_t)_{t\in\N}$ and $(\la_t)_{t\in\N}$ are chosen as in  
(\ref{choice}).

Firstly, Lemmas \ref{lem:fla} and \ref{lem:ftk} provide deterministic upper bounds. Then the rest of the Appendix aims at obtaining probabilistic bounds of $(f_t)_{t\in\N}$, based on a decomposition of  $f_t-f_\rho$ into two parts in (\ref{dec:ft}): $g_t$  is purely deterministic and is upper bounded  in $\L2$-norm in Lemma \ref{ubgt}, and  $h_t$ is studied in detail in Lemmas \ref{itub} and following. Lemma \ref{loght} yields logarithmic estimates with large probability.

\begin{lem} \label{lem:fla}For any $\lambda>0$,

(A) $\|f_\lambda\|_K \leq M_\rho/\sqrt{\lambda}$;

(B) $\|f_\lambda\|_\rho \leq M_\rho$.
\end{lem}
\begin{proof}
(A) By definition,
\[ f_\lambda = \arg \min_{f\in \H_K} \|f - f_\rho\|_\rho^2 + \lambda \|f\|^2_K. \]
The term we minimize on the right-hand side takes the value $\|f_\rho\|_\rho^2$ at $f=0$, so that 
\begin{equation} \label{eq:varbnd1}
\|f_\lambda - f_\rho\|_\rho^2 + \lambda \|f_\lambda\|^2_K \leq \|f_\rho\|_\rho^2 \leq M_\rho^2,
\end{equation}
which yields the result.

(B) Using (\ref{expfla}), 
\[ \|f_\lambda\|_\rho=\|(L_K+\la I)^{-1} L_K f_\rho \|_\rho\leq \|(L_K+\la I)^{-1} L_K\| \|f_\rho\|_\rho \leq \|f_\rho\|_\rho \leq M_\rho. \]
\end{proof}

\begin{lem} \label{lem:ftk} Assume $t_0^\tet\geq a(\ka^2+b)$. Then, for all $t\in \N$,
\[ \|f_t\|_K \leq \frac{\ka \M}{\la_t}. \]
\end{lem}

\begin{proof}
Recall that $f_t = (I-\g_tA_t)f_{t-1}+\g_ty_tK_{x_t}$.
Now assume $t_0^\tet\geq\ka^2+1$: using (\ref{bdat}), 
\[ \|f_t \|_K \leq \|1 -\ga_tA_t\| \|f_{t-1}\|_K + \ga_t \|y_t K_{x_t}\|_K \leq (1-\ga_t \la_t) \|f_{t-1}\|_K + \ga_t \ka \M. \]
By induction on $t$, we deduce
\beq
\|f_t \|_K \leq\ka \M \sum_{j=1}^t \ga_j \prod_{i=j+1}^t (1 - \ga_i \la_i)
 \leq \max_{1\leq j\leq t}(\frac{1}{\la_j}) \sum_{j=1}^t \ga_j\la_j
\prod_{i=j+1}^t (1 - \ga_i \la_i )\leq \frac{1}{\la_t}, 
\eeq
since
$$\sum_{j=1}^t \ga_j\la_j  \prod_{i=j+1}^t (1 - \ga_i \la_i ) = 1 - \prod_{i=1}^t (1 - \ga_i \la_i ). $$
\end{proof}

In the rest of Appendix, we prove probabilistic bounds of $(f_t)_{t\in\N_0}$. First observe that the definition of the online learning sequence (\ref{eq:ft}) can be rewritten as 
$$f_t-f_{\rho}=[I-\g_t(L_t+\la_tI)](f_{t-1}-f_{\rho})+\g_t(y_tK_{x_t}-L_tf_\rho)-\g_t\la_tf_\rho.$$

Let us now define the following $(\F_t)_{t\in\N_0}$-adapted processes $(g_t)_{t\in\N_0}$ and $(h_t)_{t\in\N_0}$ recursively by
$$g_0:=-f_\rho,~~h_0:=0,$$
and 
\begin{eqnarray*}
g_t:&=&[I-\g_t(L_K+\la_tI)]g_{t-1}-\g_t\la_tf_\rho,\\
h_t:&=&[I-\g_t(L_t+\la_tI)]h_{t-1}+\g_t(y_tK_{x_t}-L_tf_\rho)+\g_t(L_K-L_t)g_{t-1}.
\end{eqnarray*}
We can easily prove by induction that 
\begin{equation}
\label{dec:ft}
f_t-f_\rho=g_t+h_t,
\end{equation}
using $f_0=0$. 
\bl
\label{ubgt}
Assume $t_0^\tet\geq a(\ka^2+b)$. Then, for all $t\in\N_0$, 

(A) $\|g_t\|_\rho\leq\ M_\rho$;

(B) $\|g_t+f_\rho\|_K\leq3M_\rho/\sqrt{\la_t}$.
\el
\bp
We prove (A) by induction: $\|g_0\|_\rho=\|f_\rho\|_\rho\leq M_\rho$ and, for all $t\in\N$, if we assume $\|g_{t-1}\|_\rho\leq M_\rho$ then, using (\ref{bdaht}), 
\begin{eqnarray*}
\|g_t\|_\rho\leq\|I-\g_t(L_K+\la_tI)]\| \|g_{t-1}\|_\rho+\g_t\la_t\|f_\rho\|_\rho
\leq(1-\g_t\la_t)\|g_{t-1}\|_\rho+\g_t\la_tM_\rho\leq M_\rho.
\end{eqnarray*}

To prove (B), observe that, for all $t\in\N$, 
\beq
g_t+f_\rho&=&[I-\g_t(L_K+\la_tI)]g_{t-1}+(1-\g_t\la_t)f_\rho
=[I-\g_t(L_K+\la_tI)](g_{t-1}+f_\rho)+\g_tL_Kf_\rho\\
&=&[I-\g_t(L_K+\la_tI)](g_{t-1}+f_\rho)+\g_t(L_K+\la_tI)f_{\la_t},
\eeq
so that
$$g_t+f_\rho-f_{\la_t}=[I-\g_t(L_K+\la_tI)](g_{t-1}+f_\rho-f_{\la_t}).$$

Let, for all $t\in\N$, 
$$w_t:=g_t+f_\rho-f_{\la_t}.$$
Then it is easy to show by induction that
$$w_t=\PPi_1^tw_0+\sum_{k=1}^{t}\PPi_k^{t}(f_{\la_k}-f_{\la_{k-1}})$$
which implies, using Theorem \ref{thm:pie1} (D) with $r=0$, and Lemma \ref{lem:fla} (A) ($w_0=-f_{\la_0}$) that
$$\|w_t\|\le2M_\rho/\sqrt{\la_t}.$$
This enables us to conclude, using again Lemma \ref{lem:fla} (A).
\ep

For all $t\in\N_0$ and $M\in\R_+\cup\{\iy\}$, let 
\beq
&\oL_t:=\1_{\{|h_{t-1}(x_t)|\leq M\}}L_t, ~~ &\tL_t:=\1_{\{|h_{t-1}(x_t)|>M\}}L_t,\\
&\oL_K:=\E_{t-1}[\oL_t],~~&\tL_K:=\E_{t-1}[\tL_t].
\eeq
Note that $L_t=\oL_t+\tL_t$ and $L_K=\oL_K+\tL_K$.

For all $t\in\N$, let 
\begin{eqnarray}
\label{lht}
\oh_t&:=&[I-\g_t(\oL_t+\la_tI)]h_{t-1}+\g_t(y_tK_{x_t}-L_tf_\rho)+\g_t(L_K-L_t)g_{t-1}=h_t+\g_t\tL_th_{t-1}\\
\label{lkt1}
k_t&:=&\oh_t-(1-\g_t\la_t)h_{t-1}=\g_t[-\oL_th_{t-1}+(y_tK_{x_t}-L_tf_\rho)+(L_K-L_t)g_{t-1}]\\
\label{lkt2}
&=&\g_t[-\oL_th_{t-1}+y_tK_{x_t}+L_Kg_{t-1}-L_t(f_\rho+g_{t-1})].
\end{eqnarray}

In Lemma \ref{itub} we upper bound $\|\oh_t\|_K^2$ in conditional expectation; note that the result still holds when $M=\iy$. We threshold $h_t$ into $\oh_t$ in order to limit its conditional variance, which will be necessary in order to obtain logarithmic estimates with large probability in Lemma \ref{loght}, using on the other hand Lemma \ref{hleoh} showing that, if $M$ is large enough, $\|h_t\|_K\leq\|\oh_t\|_K$.
\bl
\label{itub}
Assume $t_0^\theta\geq 2a(b+\ka^2)$. 
For all $M\in\R_+\cup\{\iy\}$ and $t\in\N$, 
$$\E_{t-1}[\|\oh_t\|_K^2]\leq(1-\g_t\la_t)^2\|h_{t-1}\|_K^2+3\ka^2M_\rho^2\g_t^2.$$

In particular, assume moreover that $\tet\ge1/2$, $\e:=ab-(\tet-1/2)>0$ and $t_0\ge\max(2ab,2\e,\e+(2\tet-1)/\e)$, and let  $A:=a\ka M_\rho\sqrt{3/\e}$. Then $\|h_{t-1}\|_K\ge A\ot^{1/2-\tet}$ implies 
$$\ot^{\tet-1/2}\E_{t-1}[\|\oh_t\|_K]\leq (\ot-1)^{\tet-1/2}\|h_{t-1}\|_K.$$
\el
\bp
For all $t\in\N$, let
$$\ze_t:=(\oL_K-\oL_t)h_{t-1}+(L_K-L_t)g_{t-1}+(y_tK_{x_t}-L_tf_\rho),$$
so that
$$\oh_t=[I-\g_t(\oL_K+\la_tI)]h_{t-1}+\g_t\ze_t.$$

Using $\E_{t-1}[\ze_t]=0$, we deduce that
\begin{equation}
\label{htflat}
\E_{t-1}[\|\oh_t\|_K^2]=\|[I-\g_t(\oL_K+\la_tI)]h_{t-1}\|_K^2+\g_t^2 E_{t-1}[\|\ze_t\|_K^2].
\end{equation}
Let us now upper bound the two summands in the right-hand side of equality (\ref{htflat}). First, 
\beq
&&\|[I-\g_t(\oL_K+\la_tI)]h_{t-1}\|_K^2\\
&&=(1-\g_t\la_t)^2\|h_{t-1}\|_K^2-2\g_t(1-\g_t\la_t)\E_{t-1}[|h_{t-1}(x_t)|^2\1_{\{|h_{t-1}(x_t)|\leq M\}}]\\
&&\,\,\,\,\,\,\,\,\,\,\,\,\,\,\,\,\,\,\,\,\,\,\,\,\,\,\,\,\,\,\,\,\,\,\,\,\,\,\,\,\,\,\,\,\,\,\,\,\,\,\,\,\,\,\,\,\,\,\,\,\,\,\,\,\,\,\,\,\,\,\,\,\,\,\,\,\,\,\,\,\,\,\,\,\,\,\,\,\,\,\,\,\,\,\,\,\,\,\,\,\,\,\,\,\,\,\,\,\,\,\,\,\,\,\,\,\,\,\,\,\,\,\,\,\,\,\,\,
+\g_t^2\|\E_{t-1}[\oL_th_{t-1}]\|_K^2,
\eeq
and
\beq
\|\E_{t-1}[\oL_th_{t-1}]\|_K^2&\leq&\left(\E_{t-1}[|h_{t-1}(x_t)|\1_{\{|h_{t-1}(x_t)|\leq M\}}\|K_{x_t}\|_K]\right)^2\\
&\leq&\ka^2\E_{t-1}[|h_{t-1}(x_t)|^2\1_{\{|h_{t-1}(x_t)|\leq M\}}],
\eeq
using conditional Jensen's inequality.

Second, using that $\E[y_tK_{x_t}-L_tf_\rho~|~\s(\F_{t-1},x_t)]=0$, 
\beq
\E_{t-1}[\|\ze_t\|_K^2]
&=&\E_{t-1}[\|(\oL_K-\oL_t)h_{t-1}+(L_K-L_t)g_{t-1}\|_K^2+\|y_tK_{x_t}-L_tf_\rho\|_K^2]\\
&\leq&\E_{t-1}[\|\oL_th_{t-1}+L_tg_{t-1}\|_K^2+\|y_tK_{x_t}\|_K^2]\\
&\leq&\ka^2\left(2\E_{t-1}[|h_{t-1}(x_t)|^2\1_{\{|h_{t-1}(x_t)|\leq M\}}]+2\|g_{t-1}\|_\rho^2+M_\rho^2\right)\\
&\leq&\ka^2\left(2\E_{t-1}[|h_{t-1}(x_t)|^2\1_{\{|h_{t-1}(x_t)|\leq M\}}]+3M_\rho^2\right),
\eeq
where we use Lemma \ref{ubgt} (A) in the last inequality.

In summary, we obtain that 
\beq
\E_{t-1}[\|\oh_t\|_K^2]&\leq&
(1-\g_t\la_t)^2\|h_{t-1}\|_K^2-\g_t(2-2\g_t\la_t-3\g_t\ka^2)\E_{t-1}[|h_{t-1}(x_t)|^2\1_{\{|h_{t-1}(x_t)|\leq M\}}]\\
&&+3\ka^2M_\rho^2\g_t^2.
\eeq
Now, the assumption $t_0^\theta\geq 2a(b+\ka^2)$ implies $2-2\g_t\la_t-3\g_t\ka^2\geq0$ for all $t\in\N$, which completes the proof of the first statement.

Let us now prove the second statement: 
\begin{eqnarray}
\nonumber
&&\De_t:=\Es_{t-1}\left[\left(1-\frac{1}{\ot}\right)^{1-2\tet}\|\oh_t\|^2-\|h_{t-1}\|^2\right]\\
\label{1dec}
&&\le\left(1-\frac{1}{\ot}\right)^{1-2\tet}\left(1-\frac{ab}{\ot}\right)^{2}\|h_{t-1}\|^2 - \|h_{t-1}\|^2 +3\ka^2M_\rho^2 a^2\ot^{-2\tet}.
\end{eqnarray}
Now, since $t_0\ge\max(2ab,2\e,\e+(2\tet-1)/\e)$ and $\tet\ge1/2$, we have  
\beq
\log\left[\left(1-\frac{1}{\ot}\right)^{1-2\tet}\left(1-\frac{ab}{\ot}\right)^{2}\left(1-\frac{\e}{\ot}\right)^{-1}\right]
\le-\frac{\e}{\ot}+\frac{2\tet-1+\e^2}{\ot^2}\le0.
\eeq
using $\log(1-x)\le-x$ for all $x\in[0,1]$ and $\log(1-x)\ge-x-x^2$ for all $x\in[0,1/2]$.

Therefore \eqref{1dec} implies
$$\De_t\le-\frac{\e}{\ot}\|h_{t-1}\|^2+3\ka^2M_\rho^2 a^2\ot^{-2\tet}\le0.$$
The conclusion follows by conditional Jensen's inequality.
\ep
\bl
\label{dif}
Assume $t_0^\tet\geq a(\ka^2+b)$ and $t_0^{1-\tet}\geq b(2+MM_\rho^{-1})$; then
$$\|k_t\|_K\le2\ka M_\rho ab^{-1}\ot^{1-2\tet}\tx{ and }\Es_{t-1}[\|k_t\|_K^2]\leq9\g_t^2M_\rho^2\ka^2.$$
\el
\bp
By definition \eqref{lkt2}, using $\|K_{x_t}\|_K\le\ka$, $\|L_Kf_\rho\|_K\leq\ka\|L_K^{1/2}f_\rho\|_K=\ka\|f_\rho\|_\rho\leq\ka M_\rho$ and Lemma \ref{ubgt} (A)-(B), we deduce
$$\|k_t\|_K\le\g_t\left[\ka(M+2M_\rho)+\|L_t(f_\rho+g_{t-1})\|_K\right]\le\ka\g_t\left(M+2M_\rho+\frac{M_\rho}{\la_t}\right)\le
\frac{2\ka\g_t M_\rho}{\la_t}=\frac{2\ka M_\rho ab^{-1}}{\ot^{2\tet-1}},$$
where we use $t_0^{1-\tet}\geq b(2+MM_\rho^{-1})$ in the last inequality.

Now, using \eqref{lkt1}, we obtain
\beq
\Es_{t-1}[\|k_t\|_K^2]&\le&3\g_t^2 \left[\Es_{t-1}[\|\oL_th_{t-1}\|_K^2]+\Es_{t-1}[\|y_tK_{x_t}\|_K^2]+\Es_{t-1}[\|L_tg_{t-1}\|_K^2]\right]\\
&\le&3\g_t^2 [2M_\rho^2\ka^2+\|g_{t-1}\|_\rho\ka^2]\le9
\g_t^2 M_\rho^2\ka^2.
\eeq
\ep
\bl
\label{hleoh}
For all $t\in\N$, assume $M\ge M_t:=4\ka M_\rho ab^{-1}\ot^{1-2\tet}$, $t_0^\tet\geq 2a(\ka^2+b)$ and $t_0^{1-\tet}\geq b(2+MM_\rho^{-1})$; then
$$\|h_t\|_K\leq\|\oh_t\|_K.$$
\el
\bp
Assume $h_{t-1}(x_t)\ge M_t$ for instance; the other case is similar. By definition, 
$$h_t=\oh_t-\g_th_{t-1}(x_t)K_{x_t}$$
so that 
$$\|h_t\|_K^2=\|\oh_t\|_K^2-2\g_th_{t-1}(x_t)\oh_t(x_t)+\g_t^2(h_{t-1}(x_t))^2K(x_t,x_t)\le \|\oh_t\|_K^2$$
if $\oh_t(x_t)\ge\ka^2\g_th_{t-1}(x_t)/2$.

But, using Lemma \ref{dif},
\beq
\oh_t(x_t)=(1-\g_t\la_t)h_{t-1}(x_t)+k_t (x_t) 
\ge(1-\g_t\la_t)h_{t-1}(x_t)-2\ka M_\rho ab^{-1}\ot^{1-2\tet}\ge\ka^2\g_th_{t-1}(x_t)/2
\eeq
if 
$$2\ka M_\rho ab^{-1}\ot^{1-2\tet}\le h_{t-1}(x_t)/2\le h_{t-1}(x_t)(1-\g_t\la_t-\ka^2\g_t/2),$$
since the assumption $t_0^\tet\geq 2a(\ka^2+b)$ implies $1-\g_t\la_t-\ka^2\g_t/2\ge1/2$.
\ep
The following logarithmic upper bound holds under the assumptions $ab-(\tet-1/2)>0$, $\tet\in[1/2,1]$ and $t_0$ sufficiently large, but we assume $b=a^{-1}$ in its statement, for notational reasons. 
\begin{cor}
\label{loght}
Assume $\tet\in[1/2,1]$, $b=a^{-1}$, $t_0^\tet\ge2+8\ka^{2}a$ and $t_0^{1-\tet}\ge4b$. Then,  with probability at least $1-\de$, 
$$\sup_{0\le k\le t}\|h_{k}\|_K(k+t_0+1)^{\tet-1/2}\le\ka M_\rho a\left[12at_0^{1/2-\tet}+15\sqrt{\log \ot}\right]\log\frac{2}{\de}.$$
\end{cor}
\bp
Let us first check that the assumptions of Lemmas \ref{itub}, \ref{dif} and \ref{hleoh} are satisfied, and apply these lemmas: $t_0^\tet\ge3+8\ka^{2}a\ge 2a(\ka^2+b)$. Now $\e=ab-(\tet-1/2)\in[1/2,1]$, and the hypothesis $t_0\ge\max(2ab,2\e,\e+(2\tet-1)/\e)$ is satisfied as long $t_0\ge3$, which is assumed here. 
We choose $M=M_t=4\ka M_\rho ab^{-1}\ot^{1-2\tet}$; now $t_0^\tet\ge8\ka a$ and $t_0^{1-\tet}\ge4b$ imply $t_0^{1-\tet}\ge b(2+4\ka ab^{-1} t_0^{1-2\tet})\ge b(2+M_tM_\rho^{-1})$. 

For all $i\in\N$, if $\|h_{i-1}\|_K\ge A(i+t_0)^{1/2-\tet}$, $A:=a\ka M_\rho\sqrt{3/\e}$, then
\begin{eqnarray*}
\|h_i\|_K & \le & \|\bar{h}_i\|_K,  \ \ \ \ \ \ \ \ \ \ \ \ \ \ \ \ \ \ \ \ \ \ \ \ \ \  \ \ \ \ \ \ \ \ \ \  \ \ (\text{Lemma \ref{hleoh}} )\\
& \le & \|\bar{h}_i\|_K + \Es_{i-1}(\|\bar{h}_i\|_K) - \Es_{i-1} (\|\bar{h}_i\|_K) \\ 
& \le & \left(1-\frac{1}{i+t_0}\right)^{\tet-1/2}\|h_{i-1}\|_K+\e_i, \ \ \ \ \  (\text{Lemma \ref{itub}} )\\
\end{eqnarray*}
where
$$\eps_i:=\|\oh_i\|_K-\Es_{i-1}(\|\oh_i\|_K)$$
satisfies
$$\|\e_i\|_K\le4\ka M_\rho ab^{-1}(i+t_0)^{1-2\tet}\tx{ and }\Es_{i-1}[\|\e_i\|^2]\leq9\g_i^2M_\rho^2\ka^2.$$
Let, for all $i\in\N$, 
$$\eta_i:=\sum_{k=1}^i\e_k(k+t_0)^{\tet-1/2}\1_{\{\|h_{k-1}\|_K\ge A(k+t_0)^{1/2-\tet}\}}.$$

Fix $t\in\N$. For all $0\le i<t$, $\|\eta_{i+1}-\eta_i\|\le4\ka M_\rho a^2t_0^{1/2-\tet}$, and 
$$\sum_{k=1}^t\Es_{k-1}\|\eta_k\|^2\le9\ka^2 M_\rho^2a^2\sum_{k=1}^t(k+t_0)^{-1}\le9\ka^2 M_\rho^2a^2\log(1+t/t_0).$$
Let $$\De:=\left\{\sup_{1\le i\le t}\left\|\eta_i\right\|\le2\ka M_\rho a\left[\frac{4at_0^{1/2-\tet}}{3}+3\sqrt{\log\left(1+\frac{t}{t_0}\right)}\right]\log\frac{2}{\de}\right\}.$$
By Proposition \ref{prop:bernstein}, $P(\De)\ge1-\de$. 

Now assume $\De$ holds. Let, for all $k\in\N$, $$x_k:=\|h_{k}\|_K(k+t_0+1)^{\tet-1/2}.$$

For all $k\le t$, let
$$m:=\max\{j\le k\,:\, \|h_j\|_K<A(j+t_0+1)^{1/2-\tet}\}.$$
If $m<k$, then 
\beq
x_{m+1}&\le&[A(m+t_0+1)^{1/2-\tet}+2\ka M_\rho a^{2}(m+t_0+1)^{1-2\tet}](m+t_0+2)^{\tet-1/2}\\
&\le&\frac{\sqrt{5}}{2}[A+2\ka M_\rho a^{2}t_0^{1/2-\tet}]\le\frac{\sqrt{5}}{2}[\sqrt{6}a\ka M_\rho+2\ka M_\rho a^{2}t_0^{1/2-\tet}];
\eeq
the second inequality comes from $[(m+t_0+2)/(m+t_0+1)]^{\tet-1/2}\le\sqrt{5/4}$, since $t_0\ge3$.

On the other hand it is easy to prove by induction that, for all $k\le t$, 
$$x_k\le x_{m+1}+\eta_k-\eta_{m+1}$$
and, therefore,
\beq
x_k
&\le&\ka M_\rho a\left[\left(\sqrt{5}+\frac{16}{3}\right)at_0^{1/2-\tet}+\frac{\sqrt{30}}{2}+12\sqrt{\log\left(1+\frac{t}{t_0}\right)}\right]\log\frac{2}{\de}\\
&\le&12\ka M_\rho a\left[at_0^{1/2-\tet}+\frac{1}{4}+\sqrt{\log\left(1+\frac{k}{t_0}\right)}\right]\log\frac{2}{\de}\\
&\le&\ka M_\rho a\left[12at_0^{1/2-\tet}+15\sqrt{\log t}\right]\log\frac{2}{\de},
\eeq
using in the last inequality that,  for all $t\ge1$ and $t_0\ge2$, 
$$\frac{1}{4}+\sqrt{\log(1+t/t_0)}\le\frac{1}{4}+\sqrt{\log (t+t_0)}\le\frac{5}{4}\sqrt{\log (t+t_0)}.$$
\ep


\section*{Appendix C: Proof of Results of Section \ref{sec:thmconv}}
\renewcommand{\thesection}{C}
\setcounter{equation}{0} \setcounter{thm}{0}
\renewcommand{\thethm}{C.\arabic{thm}}
\renewcommand{\theequation}{C-\arabic{equation}}
\begin{proof}[Proof of Lemma \ref{lem:onebnd}]
Assume $t\ge j_0$. The spectral Theorem for compact operators implies that there is an orthonormal basis of $\W$ consisting of eigenvectors of $A_t$, so that, if $(\al_{t,k})_{k\in\N}$ are the eigenvalues of $A_t$, then
$$\|A_t^{-1}\|^{-1}=\min_{k\ge1}\al_{t,k}\ge\amin_t,~~\|I-\g_tA_t\|=\max_{k\ge1}(1-\g_t\al_{t,k})\ge0,$$
where we use that, for all $k\in\N$, $\g_t\al_{t,k}\le\g_t\amax_t\le1$. 

But $\min_{k\ge1}\al_{t,k}\ge\amin_t$ implies 
$\max_{k\ge1}(1-\g_t\al_{t,k})\le1-\g_t\amin_t$, thus $(A)$. 

The last claim follows from the inequality
$$\prod_{i=j}^t\left(1-\frac{c}{i+\t}\right)\leq\exp\left(-\sum_{i=j}^t\frac{c}{i+\t}\right)\leq
\exp\left(-c\log\left(\frac{\ot+1}{j+\t}\right)\right)=\left(\frac{j+\t}{\ot+1}\right)^c.$$

\end{proof}

\begin{proof}[Proof of Theorem \ref{thm:convergence}]
First, $\g_t\amax_t\to0$ implies that there exists $j_0\in\N$ such that $\g_t\amax_t\le1$ for all $t\ge j_0$.  Hence Lemma \ref{lem:onebnd} $(B)$ applies, so that 
\begin{equation}
\label{ubpr}
\|\Pi_j^t\|\le\prod_{i=j}^t(1-\g_i\amin_i).
\end{equation}

Let us use the reversed martingale decomposition of $r_.$, from times $j_0$ to $t$:
$$\|r_{t}\|\le\Err_{init}(t)+\Err_{samp}(t)+\Err_{drift}(t),$$
where
$$\Err_{init}(t):=\|\Pi_{j_0+1}^t r_{j_0}\|,~~\Err_{samp}(t):=\|\sum_{j=j_0+1}^t \gamma_j \Pi_{j+1}^t (A_j \w_j - b_j)\|,~~
\Err_{drift}(t):=\|\sum_{j=j_0+1}^t \Pi_{j}^t \Delta_j\|.$$
Now, by (\ref{ubpr}), 
$$\E(\Err_{init}(t)^2)\le\exp\left(-2\sum_{i=j_0+1}^t\g_i\amin_i\right)\E(\|r_{j_0}\|^2)\to_{t\to\iy}0$$
since $\sum_{t} \gamma_t \amin_t = \infty$, and
$$\Err_{drift}(t)\le\sum_{j=j_0+1}^t\|\De_j\|\prod_{i=j}^n(1-\g_i\amin_i)\to_{t\to\iy}0$$
by assumption $(C')$. Now consider the sample error. Using the independence of $(z_t)_{t\in\N}$, 
\begin{eqnarray*}
\E(\Err_{samp}(t)^2)&&=\E \left\| \sum_{j=j_0+1}^t \gamma_j \Pi_{j+1}^t (A_j \w_j - b_j) \right \|^2 = \sum_{j=j_0+1}^t \ga_j^2 \E\|\Pi_{j+1}^t (A_j \w_j - b_j)\|^2\\
&&\leq C \sum_{j=j_0+1}^t \ga_j^2 \prod_{i=j+1}^t (1 -\g_i\amin_i)^2,
\end{eqnarray*}
where 
$C:=\sup_{t\in\N}\E \|A_t \w_t - b_t \|^2<\iy$ by assumption. This completes the proof, using $(C')$. 
\end{proof}

\begin{proof}[Proof of Lemma \ref{lem:sums}]
Let $\e>0$. The assumptions $\limsup_{t\to\iy} a_t/b_t=0$ and $b_t\to_{t\to\iy}0$ imply that there exists $N\in\N$ such that $a_t\le\e b_t/2$ and $b_t\le1$ for all $t>N$. On the other hand, $\sum_{t\in\N}b_t=\iy$ implies that there exists $N_1\in\N$ such that, for all $n\ge N_1$, 
$$\sum_{k=1}^N a_k\prod_{i=k+1}^n(1-b_i)<\frac{\e}{2}.$$
Now
$$\sum_{k=N+1}^n a_k\prod_{i=k+1}^n(1-b_i)\le\frac{\e}{2}
\sum_{k=N+1}^n b_k\prod_{i=k+1}^n(1-b_i),$$
and we can write the right-hand side of this last inequality as a telescopic sum, i.e.
$$\sum_{k=N+1}^n b_k\prod_{i=k+1}^n(1-b_i)=
\sum_{k=N+1}^n [1-(1-b_k)]\prod_{i=k+1}^n(1-b_i)
=1-\prod_{i=N+1}^n(1-b_i)\le1,$$
which enables us to conclude.
\end{proof}

\bibliographystyle{chicagoc}

\end{document}